\documentclass[11pt,oneside,english]{article}
\usepackage[T1]{fontenc}
\usepackage[latin9]{inputenc}
\usepackage{verbatim}
\usepackage{mathrsfs}
\usepackage{amsmath}
\usepackage{amstext}
\usepackage{amsthm}
\usepackage{amssymb}
\usepackage{stmaryrd}
\usepackage{stackrel}
\usepackage{color}
\usepackage[all]{xy}
\usepackage{anysize}

\makeatletter

\numberwithin{equation}{section}
\numberwithin{figure}{section}
\theoremstyle{plain}
\newtheorem{thm}{\protect\theoremname}[section]
  \theoremstyle{plain}
  \newtheorem{prop}[thm]{\protect\propositionname}
  \theoremstyle{definition}
  \newtheorem{defn}[thm]{\protect\definitionname}
  \theoremstyle{plain}
  \newtheorem{cor}[thm]{\protect\corollaryname}
  \theoremstyle{plain}
  \newtheorem{lem}[thm]{\protect\lemmaname}
  \theoremstyle{remark}
  \newtheorem{rem}[thm]{\protect\remarkname}
  \theoremstyle{definition}
  
  \theoremstyle{plain}
  \newtheorem*{lem*}{\protect\lemmaname}

\makeatother

\usepackage{babel}
  \providecommand{\corollaryname}{Corollary}
  \providecommand{\definitionname}{Definition}
  \providecommand{\examplename}{Example}
  \providecommand{\lemmaname}{Lemma}
  \providecommand{\propositionname}{Proposition}
  \providecommand{\remarkname}{Remark}
\providecommand{\theoremname}{Theorem}

\marginsize{3cm}{3cm}{2.3cm}{2.3cm}

\begin{document}

\title{{\bf A hyperspace of convex bodies arising from
 tensor norms}}

\author{{\bf Luisa F. Higueras-Monta{\~n}o}
\\ {\it Centro de Investigaci\'on en Matem\'aticas (CIMAT)}\\ {\it A.P. 402 Guanajuato, Gto., M\'exico}\\ {\it fher@cimat.mx}
}

\date{ }
\maketitle
\begin{abstract}
In a preceding work it is determined when a centrally symmetric convex body in $\mathbb{R}^d,$ $d=d_1\cdots d_l,$ is the closed unit ball of a reasonable crossnorm on $\mathbb{R}^{d_1}\otimes\cdots\otimes\mathbb{R}^{d_l}.$
Consequently, the class of tensorial bodies is introduced, an associated tensorial Banach-Mazur distance is defined and the corresponding Banach-Mazur type compactum is proved to exist. In this paper, we introduce the hyperspace of these convex bodies. We called ``the space of tensorial bodies''. It is proved that the group of linear isomorphisms on $\mathbb{R}^{d_1}\otimes\cdots\otimes\mathbb{R}^{d_l}$ preserving decomposable vectors acts properly (in the sense of Palais) on it. A convenient compact global slice for the space is constructed. With it, topological representatives for  the space of tensorial bodies and the Banach-Mazur type compactum are given. Among others, it is showed that the set of ellipsoids in the class of tensorial bodies is homeomorphic to the Euclidean space of dimension $p=\frac{d_1(d_1+1)}{2}+\cdots+\frac{d_l(d_l+1)}{2}.$ We also prove  that both the projective and the injective tensor products of $0$-symmetric convex bodies are continuous functions with respect to the Hausdorff distance. 
\newline

\noindent{{\it Keywords:} Convex body, Tensor norm, Hyperspace, Lie groups, Proper actions, Tensor product of convex sets, Linear mappings on tensor spaces,  Banach-Mazur compactum.}
\newline

\noindent{{\it 2000 Mathematics Subject Classification:} 57N20, 46M05, 52A21, 57S20,  15A69.}

\end{abstract}

\section{Introduction}

Nowadays, tensor products appear as basic tools in many problems of both pure and applied nature, as can be seen in \cite{DeSilva2008,Kolda2009} and the monograph \cite{Landsberg2019}. In this sense, the theory of tensor products of Banach spaces, established by A. Grothendieck \cite{Grothendieck}, has become in an essential tool for the study of tensor products and its applications. This can be traced by its influence to a wide range of areas, from Banach space theory \cite{defantfloret,Diestelfourie,Ryan2013} to Mathematical Analysis or Graph Theory \cite{pisier} and theoretical computer science \cite{KhotNaor}.

In \cite{tensorialbodies}, M. Fern\'andez-Unzueta and the author give an intrinsic description of the convex bodies associated to tensor norms on finite dimensions. It means characterizing when a $0$-symmetric convex body $Q$ on $\mathbb{R}^d,$ $d=d_1\cdots d_l,$ is the unit ball of a reasonable crossnorm on the tensor space $\otimes_{i=1}^l(\mathbb{R}^{d_i},\|\cdot\|_i)$ for some norms $\|\cdot\|_i$ not determined \textit{a priori}. See \cite[Theorem 3.2]{tensorialbodies}. This class of convex bodies is called \textit{tensorial bodies} \cite[Definition 3.3]{tensorialbodies}. Among its principal properties, a Banach-Mazur type distance is introduced and its associated Banach-Mazur type compactum is exhibited \cite[Theorem 3.13]{tensorialbodies}. 
In the present work, we introduce the hyperspace consisting of the tensorial bodies and investigate its topological structure. The main results provide topological representatives for both the space of tensorial bodies and its associated Banach-Mazur type compactum,  see Corollaries \ref{cor:topol rep for BM}, \ref{cor:homeomorfismo importante} and \ref{cor: tensorial bodies and Rp}.


The Banach-Mazur compactum is a central object 
in the study of Banach spaces  \cite{Tomczak-Jaegermann1989}. It 
lies between two areas, namely the geometry of Banach spaces and infinite dimensional topology. Many of its metric properties, such as diameters and distances between particular points, were deeply studied in the twentieth century \cite{Gluskin,JohnF,SzarekgeometryBM}. In contrast, its topological structure was not determined until the beginning of this century. It was due to the works of S. Antonyan \cite{banacamazurcompactum,antonyananerviowest} and {S}. {A}geev and D. Repov\v{s} \cite{AgeevRepovs1} that, via topological groups, the topology of the Banach-Mazur compactum was established. In this way, it is worth to notice that the results presented here extend, to the tensorial setting, many of the properties exhibited in \cite{banacamazurcompactum} for the  Banach-Mazur compactum and the hyperspace of $0$-symmetric convex bodies.

\medskip

Below we present  the contents of the paper. We begin by stating basic results and notation from tensor norms and group actions used along the work. Then,  in Section \ref{sec: tensor bodies}, we review the main properties of the tensorial bodies \cite{tensorialbodies}. A tensorial body $Q$ in $\otimes_{i=1}^l\mathbb{R}^{d_i}$ is a $0$-symmetric convex body such that 
$Q_{1}\otimes_{\pi}\cdots\otimes_{\pi}Q_{l}\subseteq Q\subseteq Q_{1}\otimes_{\epsilon}\cdots\otimes_{\epsilon}Q_{l}$
for some $0$-symmetric convex bodies $Q_i\subset\mathbb{R}^{d_i}.$ The set of tensorial bodies in $\otimes_{i=1}^l\mathbb{R}^{d_i}$ is denoted by $\mathcal{B}_\otimes(\otimes_{i=1}^l\mathbb{R}^{d_i}).$ Here, $\otimes_{\pi}$ and $\otimes_{\epsilon}$ are, respectively, the projective and the injective tensor product of $0$-symmetric convex bodies (\cite{Aubrun2006}, \cite[Section 4.1]{Aubrun2017}). 

When regarding as a hyperspace of compact convex sets in $\otimes_{i=1}^l\mathbb{R}^{d_i},$ the space of tensorial bodies $\mathcal{B}_\otimes(\otimes_{i=1}^l\mathbb{R}^{d_i})$ is  a topological subspace of the space of $0$-symmetric convex bodies in $\otimes_{i=1}^l\mathbb{R}^{d_i}\simeq\mathbb{R}^d,$ $d=d_1\cdots d_l.$ In this way, in Section \ref{sec:hyperspace of convex bodies}, we prove that this space is closed and contractible, see Propositions \ref{prop:Bsigma is closed} and \ref{prop: tensorial bodies are contractil}. 
An expected and fundamental result is given in Proposition \ref{prop:porducto projectivo e injectivo es continuo con hauss}, where we show that the projective $\otimes_\pi$ and the injective $\otimes_\epsilon$ tensor products are continuous functions with respect to the Hausdorff distance.

To go further into the topological structure of the space of tensorial bodies, in Section \ref{sec:the action of GLtensor}, we introduce an action of a Lie group on this space. We prove that  the group of linear isomorphisms on $\otimes_{i=1}^l\mathbb{R}^{d_i}$ preserving decomposable vectors $GL_\otimes(\otimes_{i=1}^l\mathbb{R}^{d_i})$ (see Section \ref{sec: tensor bodies} for the definition and basic proterties) acts properly (in the sense of Palais) on $\mathcal{B}_\otimes(\otimes_{i=1}^l\mathbb{R}^{d_i})$ (Theorem \ref{thm:the action is proper}). The action is given by
\begin{alignat*}{1}
GL_\otimes(\otimes_{i=1}^l\mathbb{R}^{d_i})\times\mathcal{B}_{\otimes}(\otimes_{i=1}^l\mathbb{R}^{d_i}) & \longrightarrow\mathcal{B}_{\otimes}(\otimes_{i=1}^l\mathbb{R}^{d_i})\\
\left(T,Q\right) & \mapsto TQ:=\left\{ Tu:u\in Q\right\} .
\end{alignat*}
This result together with \cite[Corollary 4.3]{tensorialbodies} allows us to calculate the topological structure of the set of tensorial ellipsoids $\mathscr{E}_{\otimes}(\otimes_{i=1}^{l}\mathbb{R}^{d_{i}})$ (i.e the ellipsoids in $\otimes_{i=1}^l\mathbb{R}^{d_i}$ that also are tensorial bodies, see Section \ref{sec:tensorial ellipsoids}). We show (Corollary \ref{cor:glsigma over osigma and ellipsoids}) that:
\begin{center}
 $\mathscr{E}_{\otimes}(\otimes_{i=1}^{l}\mathbb{R}^{d_{i}})$ \textit{is homeomorphic to} 
$\mathbb{R}^p$ \textit{with} $p=\frac{d_1(d_1+1)}{2}+\cdots+\frac{d_l(d_l+1)}{2}.$\\
\par\end{center}
Throughout the paper $d,d_1,\ldots,d_l\geq2$ and $l\geq2$ are integers.  
In this section, it is fundamental the subgroup $O_\otimes(\otimes_{i=1}^{l}\mathbb{R}^{d_{i}}):=O(\otimes_{i=1}^{l}\mathbb{R}^{d_{i}})\cap GL_\otimes(\otimes_{i=1}^{l}\mathbb{R}^{d_{i}})$ which consists of the orthogonal maps on the Hilbert tensor product $(\otimes_{i=1}^{l}\mathbb{R}^{d_{i}},\|\cdot\|_H)$  preserving decomposable vectors 

In Section \ref{sec:global slice}, 
we construct a compact $O_\otimes(\otimes_{i=1}^{l}\mathbb{R}^{d_{i}})$-global slice $\mathscr{L}_\otimes(\otimes_{i=1}^{l}\mathbb{R}^{d_{i}})$  for the space of tensorial bodies. To that end, we first define an equivariant retraction 
$l_\otimes:\mathcal{B}_\otimes(\otimes_{i=1}^{l}\mathbb{R}^{d_{i}})\rightarrow\mathscr{E}_\otimes(\otimes_{i=1}^{l}\mathbb{R}^{d_{i}})$ from the space of tensorial bodies onto the tensorial ellipsoids. 
It sends each tensorial body $Q$ in $\otimes_{i=1}^{l}\mathbb{R}^{d_{i}}$ to the L\"{o}wner ellipsoid  $L\ddot{o}w(Q^1\otimes_\pi\cdots\otimes_\pi Q^l)$ of the projective tensor product of the $0$-symmetric convex bodies $Q^i$ associated to $Q$ in Remark \ref{rem: canonical section} (Proposition \ref{prop:lsigma es RETRACCION}). Then, the compact $O_\otimes(\otimes_{i=1}^{l}\mathbb{R}^{d_{i}})$-global slice of $\mathcal{B}_\otimes(\otimes_{i=1}^{l}\mathbb{R}^{d_{i}})$ is the inverse image 
$\mathscr{L}_\otimes(\otimes_{i=1}^{l}\mathbb{R}^{d_{i}})=l_\otimes^{-1}(B_2^{d_1,\ldots,d_l})$ of the Euclidean ball $B_2^{d_1,\ldots,d_l}$ on $(\otimes_{i=1}^{l}\mathbb{R}^{d_{i}},\|\cdot\|_H).$ See Theorem \ref{thm:compact global slice}.

It is in Section \ref{sec:main results} where we accomplish our principal goals: to provide topological representatives for the space of tensorial bodies (Corollaries \ref{cor:homeomorfismo importante} and \ref{cor: tensorial bodies and Rp} ) and its associated Banach-Mazur type compactum $\mathcal{BM}_\otimes(\otimes_{i=1}^{l}\mathbb{R}^{d_{i}})$ (Corollary \ref{cor:topol rep for BM}).  $\mathcal{BM}_\otimes(\otimes_{i=1}^{l}\mathbb{R}^{d_{i}})$ consists of the classes of tensorial bodies determined by $GL_\otimes(\otimes_{i=1}^{l}\mathbb{R}^{d_{i}})$ endowed with the metric $log\delta_{\otimes}^{BM}$ induced by the tensorial Banach-Mazur distance $\delta_{\otimes}^{BM},$ see \cite[Theorem 3.13]{tensorialbodies}. In Corollary \ref{cor:topol rep for BM}, we first prove that $\mathcal{BM}_\otimes(\otimes_{i=1}^{l}\mathbb{R}^{d_{i}})$ is homeomorphic to  the orbit space $\mathcal{B}_\otimes(\otimes_{i=1}^{l}\mathbb{R}^{d_{i}})/GL_\otimes(\otimes_{i=1}^{l}\mathbb{R}^{d_{i}})$. Then, from Theorem \ref{thm:compact global slice}, we obtain a homeomorphism between the compactum and the orbit space $\mathscr{L}_\otimes(\otimes_{i=1}^{l}\mathbb{R}^{d_{i}})/O_\otimes(\otimes_{i=1}^{l}\mathbb{R}^{d_{i}}).$  We thus get the representatives:
\begin{center}
$\mathcal{BM}_\otimes(\otimes_{i=1}^{l}\mathbb{R}^{d_{i}})\cong\mathcal{B}_\otimes(\otimes_{i=1}^{l}\mathbb{R}^{d_{i}})/GL_\otimes(\otimes_{i=1}^{l}\mathbb{R}^{d_{i}})\cong\mathscr{L}_\otimes(\otimes_{i=1}^{l}\mathbb{R}^{d_{i}})/O_\otimes(\otimes_{i=1}^{l}\mathbb{R}^{d_{i}})$.\\
\par
\end{center}
This result extends to the tensorial setting Corollary 1 of \cite{banacamazurcompactum}  which exhibits new representatives for the Banach-Mazur compactum.


In Corollary \ref{cor:homeomorfismo importante}, we show that the space of tensorial bodies is homeomorphic to the product $\mathscr{L}_\otimes(\otimes_{i=1}^{l}\mathbb{R}^{d_{i}})\times\mathscr{E}_\otimes(\otimes_{i=1}^{l}\mathbb{R}^{d_{i}}).$ This, along with Corollary \ref{cor:glsigma over osigma and ellipsoids}, allows us to prove that
$\mathcal{B}_\otimes(\otimes_{i=1}^{l}\mathbb{R}^{d_{i}})$ is homeomorphic to $\mathscr{L}_\otimes(\otimes_{i=1}^{l}\mathbb{R}^{d_{i}})\times\mathbb{R}^p$ (Corollary \ref{cor: tensorial bodies and Rp}).

We finish the paper with an appedix (Section \ref{sec: properties of GL tensor}) about the Lie group $GL_\otimes(\otimes_{i=1}^{l}\mathbb{R}^{d_{i}})$ which may be of independent interest. We show that the subgroup $O_\otimes(\otimes_{i=1}^{l}\mathbb{R}^{d_{i}})$ is a maximal compact subgroup of $GL_\otimes(\otimes_{i=1}^{l}\mathbb{R}^{d_{i}})$ (Proposition \ref{prop: maximal comp group}). This result, together with the so called polar decomposition theorem,  allows us to determine the Lie group structure of $GL_\otimes(\otimes_{i=1}^{l}\mathbb{R}^{d_{i}}).$ See Proposition \ref{prop: charact of GL tensor}, Corollary \ref{cor: structure GL/O tensor} and Lemma \ref{lem: split thm of Gl tensor}. 

We would like to point out that this work follows the ideas presented in \cite{banacamazurcompactum,antonyananerviowest}, where the topological structure of the Banach-Mazur compactum is determined via the action induced by the general linear group $GL(d)$. 
Recent works about the topology of some hyperspaces of convex bodies are \cite{Antonyan2013,AntonyanJonardOrdonez}.






\subsection{Notation}
\label{sec:notation}

The letters $d,d_{i}$ will denote possitive integers greater than or equal to $2$. By 
$\left\langle\cdot,\cdot\right\rangle,$ $\Vert\cdot\Vert,$ $B_2^{d}$ we denote the standard scalar product on $\mathbb{R}^d$ and its associated norm and Euclidean ball. As usual the group of linear isomorphisms on $\mathbb{R}^d$ is denoted by $GL(d).$
Every compact convex set  $Q\subset\mathbb{R}^{d}$ with nonempty interior (i.e. $\text{int}(Q)\neq\emptyset$) is called a convex body. In adition, if $Q=-Q$ then $Q$ is called a $0$-symmetric convex body. We write $\mathcal{B}(d)$ to denote the set of $0$-symmetric convex bodies in $\mathbb{R}^{d}.$  For every $Q\in\mathcal{B}(d),$ its polar set, $Q^\circ,$ is a $0$-symmetric convex body, defined as $Q^{\circ}:=\left\{ y\in\mathbb{R}^{d}:{\sup}_{x\in Q}|\left\langle x,y\right\rangle|\leq1\right\}.$

\textit{The Minkowski functional} of a $0$-symmetric convex body $Q\subset\mathbb{R}^d$  is defined by
$g_{Q}\left(x\right):=\inf\left\{ \lambda>0:\lambda^{-1}x\in Q\right\},$ for $x\in\mathbb{R}^{d}.$ A well known result concerning  $0$-symmetric convex bodies is the bijection between norms on $\mathbb{R}^{d}$ and $0$-symmetric convex bodies. This result, due to H. Minkowski \cite{Minkowski1927}, is fundamental to this work and will be used frequently without making an explicit reference. It can be stated as follows: the map $\mathcal{B}\left(d\right)\rightarrow\left\{\text{norms on }\mathbb{R}^{d}\right\}$ that sends $Q$ to its Minkowski functional $g_Q(\cdot)$
is a bijection, the unit ball of $\left(\mathbb{R}^{d},g_Q\right)$ is $Q$ and $g_{Q^{\circ}}(x)=\left\Vert \left\langle \cdot,x\right\rangle :\left(\mathbb{R}^{d},g_p\right)\rightarrow\mathbb{R}\right\Vert.$  See \cite[Remark 1.7.8]{Schneider1993}.

The results about convex bodies that will be used in this paper can be found  in  \cite{Schneider1993}.
 
\subsection{Tensor norms}
\label{sec:basic tensor}
 We use standard notation from Banach space theory and tensor products. The symbols $M$, $N$ or $M_i$ will denote Banach spaces.  The closed unit ball of $M$ will be  denoted by  $B_{M}$ and  its  dual space by $M^{*}.$ We write $\mathcal{L}\left(M,N\right)$ to denote the Banach space of  bounded linear operators from $M$ to $N,$ with the usual operator norm.
 
The tensor product of $M_i,$ $i=1,\ldots,l,$ is denoted by $\otimes_{i=1}^{l}M_{i}.$ The elements of the form $x^1\otimes\cdots\otimes x^l,$ $x^i\in M_i,$ are called decomposable vectors. For each subset $A_i\subseteq M_i,$ $i=1,\ldots,l,$ $\otimes(A_1\ldots,A_l)$ is the image of $A_1\times\cdots\times A_l$ under the canonical multilinear map $\otimes.$ The decomposable vector $x_{1}^{*}\otimes\cdots\otimes x_{l}^{*},$ $x_i^{*}\in M_i^{*},$ determines a linear functional on $\otimes_{i=1}^{l}M_{i}$ which sends $x^1\otimes\cdots\otimes x^l$ to $x_1^{*}(x^1)\cdots x_l^{*}(x^l).$

A norm $\alpha\left(\cdot\right)$ on the tensor product $\otimes_{i=1}^{l}M_{i}$ is a \textit{reasonable crossnorm} if
\begin{enumerate}
\item $\alpha\left(x^{1}\otimes\cdots\otimes x^{l}\right)\leq\left\Vert x^{1}\right\Vert \cdots\left\Vert x^{l}\right\Vert $
for every $x^{i}\in M_{i},$ $i=1,...,l.$
\item For every $x_{i}^{*}\in M_{i}^{*},$ the linear functional $x_{1}^{*}\otimes\cdots\otimes x_{l}^{*}$ is bounded
and $\left\Vert x_{1}^{*}\otimes\cdots\otimes x_{l}^{*}\right\Vert \leq\left\Vert x_{1}^{*}\right\Vert \cdots\left\Vert x_{l}^{*}\right\Vert .$
\end{enumerate}
The biggest and the smallest reasonable crossnorms are the projective tensor norm $\pi(\cdot)$ and  the injective tensor norm $\epsilon(\cdot)$ respectively. They are defined as: 
\[
\pi(u):=\inf\left\{ \stackrel[i=1]{n}{\sum}\| x_{i}^{1}\| \cdots\| x_{i}^{l}\| :u=\stackrel[i=1]{n}{\sum}x_{i}^{1}\otimes\cdots\otimes x_{i}^{l}\right\}\thinspace\text{,  and}
\]
\[
\epsilon(u):=\sup\left\{ \left|x_{1}^{*}\otimes\cdots\otimes x_{l}^{*}\left(u\right)\right|:x_{i}^{*}\in B_{M_{i}^{*}}, i=1,\ldots,l\right\}
\]
for $u\in\otimes_{i=1}^{l}M_i$. An alternative description of reasonable crossnorms, stated in terms of the norms $\pi(\cdot)$ and $\epsilon(\cdot),$ is the following: a norm $\alpha(\cdot)$ is a reasonable crossnorm if
\begin{equation}
\label{eq:Caracterizacion de normars razonables cruzadas}
\epsilon\left(u\right)\leq\alpha\left(u\right)\leq\pi\left(u\right)\text{ for every }u\in\otimes_{i=1}^{l}M_{i}.
\end{equation}

If $\alpha\left(\cdot\right)$ is a reasonable crossnorm on $\otimes_{i=1}^{l}M_{i}$, $\otimes_{\alpha,i=1}^{l}M_{i}$ will denote the normed space $\left(\otimes_{i=1}^{l}M_{i},\alpha\right),$ and $M_1\hat{\otimes}_{\alpha}\cdots\hat{\otimes}_{\alpha}M_{l}$ its completion. For a deeper discussion about tensor norms we refer the reader  to  \cite{defantfloret,Diestelfourie,Grothendieck,Ryan2013}.

On the tensor product of Euclidean spaces, there is natural scalar product $\left\langle\cdot,\cdot\right\rangle _{H}.$ It endows $\otimes_{i=1}^{l}\mathbb{R}^{d_i}$ with a reasonable crossnorm  $\Vert\cdot\Vert_{H}.$  $\left\langle\cdot,\cdot\right\rangle _{H}$ is defined, on decomposable vectors, as:
 \[
 \left\langle x^{1}\otimes\cdots\otimes x^{l},y^{1}\otimes\cdots\otimes y^{l}\right\rangle _{H}:={\Pi}_{i=1}^l\left\langle x^i,y^i\right\rangle
 \]
 and it is extended to $\otimes_{i=1}^{l}\mathbb{R}^{d_i}$ by multilinearity. The closed unit ball of $\otimes_{H,i=1}^{l}\mathbb{R}^{d_{i}}$ is denoted by $B_{2}^{d_{1},\ldots,d_{l}}$. For a thorough treatment of tensor products of Euclidean spaces, we refer the reader to \cite[Section 2.5 ]{kadisonkingrose}. 

 
 \subsection{Group Actions} 
\label{sec:group actions}

 Below we present the results from topological groups that will be used in Section 4. See \cite{Bredon,Palais1} for a deeper discussion of this topic.  
 
 All topological groups and topological spaces considered are Tychonoff. A $G$-\textit{space} is a pair $\left(X,\theta\right)$ where $G$
is a topological group and $\theta$ is a continuous action of $G$
on $X.$ If $X$ is a $G$-space and $x\in X,$ then 
$
G\left(x\right):=\left\{ gx:g\in G\right\} 
$
denotes the \textit{orbit} of $x.$ By $X/G,$ we denote the orbit space.
 As usual, 
$
G_{x}:=\left\{ g\in G:gx=x\right\} 
$
 is the  \textit{stabilizer} of $G$ at $x$. 
 For every subset $S\subseteq X$ and every subgroup $H\subseteq G,$ 
$
H\left(S\right):=\left\{ hs:h\in H,s\in S\right\} 
$
is the $H$-saturation of $S.$ If $H(S)=S,$ $S$ is called $H$\textit{-invariant}.
 
 For any subgroup $H\subseteq G,$ $G/H:=\{gH: g\in G\}$ is a $G$-space  with the action induced by left translations.
 
 A continuous map $f:X\rightarrow Y$ between two $G$-spaces is called
\textit{equivariant} or a $G$\textit{-map} if $f\left(gx\right)=gf\left(x\right)$
for every $x\in X$ and $g\in G.$
 \medskip

A special class of actions over locally compact groups are the so called proper actions, introduced by R. Palais \cite{Palais2}. They enjoy many of the desirable properties of actions over compact groups, as can bee seen in \cite{Abels,Palais2}.

Let $G$ be a locally compact group and let $X$ be a Tychonoff $G$-space. The action of $G$ on $X$ is \textit{proper} (in the sense of Palais) if any $x\in X$ has a small neighborhood $V.$ Here, a subset  $S\subseteq X$ of a $G$-space is called \textit{small} if any $x\in X$ has a neighborhood $V$ such that
 $\left[S,V\right]:=\left\{ g\in G:gS\cap V\neq\emptyset\right\},$ the \textit{transporter from $S$ to $V$,} has compact closure in $G.$
 \medskip
 
Proper actions play a fundamental role for the paper. For this reason, we include some of its main properties: 
On every $G$-space $X,$ the orbit $G(x),$ $x\in X,$ is a closed subset of $X$ and  $G_{x}$ is a compact subgroup of $G$ (see \cite[Proposition 1.1.4]{Palais2}). Moreover, there is a $G$-equivariant homeomorphism between $G/G_{x}$ and $G\left(x\right),$ see \cite[Proposition 1.1.5]{Palais2}.

Let us recall the definition of a slice:

\begin{defn}
\cite[p. 305]{Palais1}
Let $X$ be a $G$-space and $H$ a closed subgroup of $G.$ An $H$-invariant
subset $S\subseteq X$ is called an $H$-\textit{slice} in $X,$ if
$G\left(S\right)$ is open in $X$ and there exists a $G$-equivariant
map $f:G\left(S\right)\rightarrow G/H$ such that $S=f^{-1}\left(eH\right).$
The saturation $G\left(S\right)$ is called a \textsl{tubular} set.
If $G\left(S\right)=X,$ then $S$ is a \textsl{global}
$H$-slice of $X.$
\end{defn}

An important and well known result about actions of compact Lie groups $G$ establishes the existence of a $G_x$-slice for any $x$ in a $G$-space $X$,
see \cite[p. 14]{Bredon}. The case of non-compact Lie groups and proper actions is deeply studied in \cite{Abels, Palais2}.

\section{Tensorial Bodies}
\label{sec: tensor bodies}

In \cite{tensorialbodies}, the problem of establishing a geometric characterization of the unit balls of tensor normed spaces on finite dimensions is addressed.  There, by means of the correspondence between $0$-symmetric convex bodies and norms on $\mathbb{R}^{d},$ $d=d_1\cdots d_l,$ a characterization of the convex bodies that are closed unit balls of reasonable crossnorms on $\otimes_{i=1}^l\left(\mathbb{R}^{d_i},\Vert\cdot\Vert_i\right),$  for some norms (not determined \textit{a priori}) on each $\mathbb{R}^{d_i}$ is exhibited (see \cite[Theorem 3.2]{tensorialbodies}). These convex bodies are the so called tensorial bodies.


Throughout the paper, $\otimes_{i=1}^{l}\mathbb{R}^{d_{i}}$ will be a Euclidean space with the scalar product $\left\langle\cdot,\cdot\right\rangle _{H}.$ 
Given a tuple $Q_i\subset\mathbb{R}^{d_i},$ $i=1,\ldots,l,$ of $0$-symmetric convex bodies, their projective $\otimes_\pi$ and injective tensor product $\otimes_\epsilon$ (\cite{Aubrun2006}, \cite[Section 4.1]{Aubrun2017}) are defined as:
\begin{equation*}
Q_{1}\otimes_{\pi}\cdots\otimes_{\pi}Q_{l}:=\text{conv}\{ x^{1}\otimes\cdots\otimes x^{l}\in\otimes_{i=1}^{l}\mathbb{R}^{d_{i}}:x^{i}\in Q_{i}, i=1,\ldots,l\}\thinspace\text{, and}
\end{equation*}
\hspace{4cm}
$
Q_{1}\otimes_{\epsilon}\cdots\otimes_{\epsilon}Q_{l}:=\left(Q_{1}^{\circ}\otimes_{\pi}\cdots\otimes_{\pi}Q_{l}^{\circ}\right)^{\circ}.
$
\medskip

When the normed spaces $(\mathbb{R}^{d_{i}},g_{Q_{i}}),$ $i=1,\ldots,l,$ are considered, the projective $\otimes_\pi$ and the injective $\otimes_\epsilon$ tensor products of $Q_1,\ldots,Q_l$ are the closed unit balls associated to the projective $\pi(\cdot)$ and the injective $\epsilon(\cdot)$ tensor norms, respectively. That is,
\begin{equation*}
B_{\otimes_{\pi,i=1}^{l}(\mathbb{R}^{d_{i}},g_{Q_{i}})}=Q_{1}\otimes_{\pi}\cdots\otimes_{\pi}Q_{l}\thinspace\text{, and} 
\end{equation*}
\begin{equation*}
B_{\otimes_{\epsilon,i=1}^{l}(\mathbb{R}^{d_{i}},g_{Q_{i}})}=Q_{1}\otimes_{\epsilon}\cdots\otimes_{\epsilon}Q_{l}.
\end{equation*}
See \cite[pp. 5-6]{tensorialbodies}.

The previous relations give us  the possibility to describe reasonable crossnorms in terms of convex bodies. Given a $0$-symmetric convex body $Q\subset\otimes_{i=1}^l\mathbb{R}^{d_i},$ $g_Q$ is a reasonable crossnorm on $\otimes_{i=1}^l\left(\mathbb{R}^{d_i},\Vert\cdot\Vert_i\right)$ if and only if
\begin{equation}
\label{eq: inclusion convex bodies rc norms}
Q_{1}\otimes_{\pi}\cdots\otimes_{\pi}Q_{l}\subseteq Q\subseteq Q_{1}\otimes_{\epsilon}\cdots\otimes_{\epsilon}Q_{l}.
\end{equation}
where $Q_i=B_{\left(\mathbb{R}^{d_i},\Vert\cdot\Vert_i\right)},$ $i=1,\ldots,l.$ In this case, for every $x^{i}\in\mathbb{R}^{d_{i}},$ $i=1,\ldots,l,$ the following hold:
\begin{align}
\label{eq:rc in gauges}
g_{Q}\left(x^{1}\otimes\cdots\otimes x^{l}\right) & = g_{Q_{1}}\left(x^{1}\right)\cdots g_{Q_{l}}\left(x^{l}\right),\\
g_{Q^{\circ}}\left(x^{1}\otimes\cdots\otimes x^{l}\right) & = g_{Q_{1}^{\circ}}\left(x^{1}\right)\cdots g_{Q_{l}^{\circ}}\left(x^{l}\right),
\end{align}
see \cite[Proposition 3.1]{tensorialbodies}. The inclusions in (\ref{eq: inclusion convex bodies rc norms}) enables the definiton of tensorial bodies:
\begin{defn}
\label{defn:tensorial body}
(\cite[Definition 3.3]{tensorialbodies})
A $0$-symmetric convex body $Q\subset\otimes_{i=1}^{l}\mathbb{R}^{d_{i}}$ is called a \textit{tensorial body in} $\mathbf{\otimes_{i=1}^{l}\mathbb{R}^{d_{i}}}$ if there exist $0$-symmetric convex bodies $Q_{i}\subset\mathbb{R}^{d_{i}},$  $i=1,...,l,$ such that (\ref{eq: inclusion convex bodies rc norms}) holds.
\end{defn}
 If $Q$ satisfies  (\ref{eq: inclusion convex bodies rc norms}), we will say that $Q$ is a \textit{tensorial body with respect to} $Q_1,\ldots,Q_l.$ The set of tensorial bodies in $\otimes_{i=1}^{l}\mathbb{R}^{d_{i}}$ is denoted by $\mathcal{B}_{\otimes}\left(\otimes_{i=1}^{l}\mathbb{R}^{d_{i}}\right).$ The set of tensorial bodies with respect to $Q_{1},...,Q_{l}$ is denoted by $\mathcal{B}_{Q_{1},\ldots,Q_{l}}\left(\otimes_{i=1}^{l}\mathbb{R}^{d_{i}}\right).$

For every non-zero decomposable vector $\mathbf{a}\in\otimes_{i=1}^{l}\mathbb{R}^{d_{i}}$ and every $0$-symmetric convex body $Q\subset\otimes_{i=1}^{l}\mathbb{R}^{d_{i}}$ , if $\mathbf{a}=a^{1}\otimes\cdots\otimes a^{l}$ then $Q_{i}^{a^{1},\ldots,a^{l}},$ $i=1,\ldots,l,$ defined as
\begin{equation}\label{eq: Q1 Ql}
Q_{i}^{a^{1},\ldots,a^{l}}:=\left\{ x^{i}\in\mathbb{R}^{d_{i}}:a^{1}\otimes\cdots\otimes a^{i-1}\otimes x^{i}\otimes a^{i+1}\otimes\cdots\otimes a^{l}\in Q\right\},
\end{equation}
is a  $0$-symmetric convex body.  The next theorem establishes when a $0$-symmetric convex body is the closed unit ball of a reasonable crossnorm for some norms, not determined a priori, on each $\mathbb{R}^{d_i}.$
\begin{thm}
\label{thm:equivalence of tensor body} (\cite[Corollary 3.4]{tensorialbodies})
Let $Q\subset\otimes_{i=1}^{l}\mathbb{R}^{d_{i}}$ be a $0$-symmetric convex body. The following are equivalent:
\begin{enumerate}
\item $Q$ is a tensorial body in $\otimes_{i=1}^{l}\mathbb{R}^{d_{i}}.$
\item There exist norms $\|\cdot\| _{i}$ on $\mathbb{R}^{d_{i}},$
$i=1,...,l,$ such that $g_{Q}$ is a reasonable crossnorm on $\otimes_{i=1}^{l}(\mathbb{R}^{d_{i}},\|\cdot\| _{i}).$
\item  For any  $a^{1}\otimes\cdots\otimes a^{l}\in\partial Q,$ 
\begin{equation*}
\label{eq:ecuacionnnnn}
 Q_{1}^{a^{1},\ldots,a^{l}}\otimes_{\pi}\cdots\otimes_{\pi}Q_{l}^{a^{1},\ldots, a^{l}}\subseteq Q
  \subseteq  Q_{1}^{a^{1},\ldots,a^{l}}\otimes_{\epsilon}\cdots\otimes_{\epsilon}Q_{l}^{a^{1},\ldots,a^{l}}.
\end{equation*}
\end{enumerate}
In this case, $g_{Q_{i}^{a^{1},\ldots,a^{l}}}(\cdot)=\frac{1}{\|a^{i}\| _{i}}\|\cdot\| _{i}$  for $i=1,\ldots,l.$
\end{thm}
\begin{rem}
\label{rem: canonical section}
To simplify many arguments in the forthcoming proofs, it is convenient to choose the convex bodies of (\ref{eq: Q1 Ql}) in a specific way: for every $0$-symmetric convex body $Q\subset\otimes_{i=1}^{l}\mathbb{R}^{d_{i}},$  $Q^{i}$  denote the convex bodies generated by $e_{1}^{d_{1}}\otimes\cdots\otimes(\lambda e_{1}^{d_{l}}),$ $\lambda=\frac{1}{g_{Q}(e_{1}^{d_{1}}\otimes\cdots\otimes e_{1}^{d_{l}})}.$ That is, $Q^i:=Q_i^{e_{1}^{d_1},\ldots,\lambda e_{l}^{d_l}}$ for $i=1,\ldots,l.$
\end{rem}

In \cite[Section 3.2]{tensorialbodies}, it is proved that there exists a Banach-Mazur type distance on the set of tensorial bodies. It is called the tensorial Banach-Mazur distance. Its existence follows from the relation between tensorial bodies and linear mappings preserving decomposable vectors that we describe below. 

A linear map $T:\otimes_{i=1}^{l}\mathbb{R}^{d_{i}}$ $\rightarrow\otimes_{i=1}^{l}\mathbb{R}^{d_{i}}$ \textit{preserves decomposable vectors} if $T\left(x^{1}\otimes\cdots\otimes x^{l}\right)$ is a decomposable vector for any $x^{i}\in\mathbb{R}^{d_i},$ $i=1,\ldots,l.$ The set of linear isomorphisms preserving decomposabe vectors is denoted by $GL_{\otimes}(\otimes_{i=1}^{l}\mathbb{R}^{d_{i}}).$

 In \cite[Corollary 2.14]{limcampocualquiera}, it is proved that for every $T\in GL_{\otimes}(\otimes_{i=1}^{l}\mathbb{R}^{d_{i}})$ and every $x^{i}\in\mathbb{R}^{d_i},$ we have:
\begin{equation}
\label{eq:tensor group}
T\left(x^{1}\otimes\cdots\otimes x^{l}\right)=T_{1}\left(x^{\sigma\left(1\right)}\right)\otimes\cdots\otimes T_{l}\left(x^{\sigma\left(l\right)}\right)
\end{equation}
where $\sigma$ is a permutation on $\left\{ 1,...,l\right\} $ and $T_{i}\in GL(d_i)$ for $i=1,\ldots,l.$ The latter along with \cite[Theorem 3.12]{tensorialbodies} allows to prove that for every $T\in GL_{\otimes}(\otimes_{i=1}^{l}\mathbb{R}^{d_{i}}):$ 
\begin{equation}
\label{eq: glotimes preserva cuerpos tensoriales}
\text{If }Q\in\mathcal{B}_{Q_{1},\ldots,Q_{l}}\left(\otimes_{i=1}^{l}\mathbb{R}^{d_{i}}\right)\text{, then }TQ\in\mathcal{B}_{T_{1}Q_{\sigma(1)},\ldots,T_{l}Q_{\sigma(l)}}\left(\otimes_{i=1}^{l}\mathbb{R}^{d_{i}}\right).
\end{equation}

In this case, for every tuple of $0$-symmetric convex bodies $Q_i\subset\mathbb{R}^{d_i},$ $i=1,\ldots,l,$ and $\alpha=\pi,\epsilon$ it holds
\begin{equation}
\label{eq:gltensor image of proj and inj tp}
T(Q_{1}\otimes_{\alpha}\cdots\otimes_{\alpha}Q_{l})= T_{1}(Q_{\sigma(1)})\otimes_{\alpha}\cdots\otimes_{\alpha}T_{l}(Q_{\sigma(l)}).
\end{equation}

The \textbf{tensorial Banach-Mazur distance} $\delta_{\otimes}^{BM}\left(P,Q\right),$ between tensorial bodies $P,Q\subset\otimes_{i=1}^{l}\mathbb{R}^{d_{i}},$ is defined as:
\begin{equation}
\label{eq:tensorial bm distance}
\delta_{\otimes}^{BM}\left(P,Q\right):=\inf\left\{ \lambda\geq1:Q\subseteq TP\subseteq\lambda Q,\text{ for }T\in GL_{\otimes}(\otimes_{i=1}^{l}\mathbb{R}^{d_{i}})\right\}.
\end{equation}
 
In \cite[Section 3.2]{tensorialbodies}, it is proved that for each pair $P,Q\in\mathcal{B}_{\otimes}\left(\otimes_{i=1}^{l}\mathbb{R}^{d_{i}}\right),$ the infimum in (\ref{eq:tensorial bm distance}) attains its value at some $\lambda\geq1$ and $T\in GL_{\otimes}\left(\otimes_{i=1}^{l}\mathbb{R}^{d_{i}}\right).$ This naturally leads to the  following equivalence relation: For every pair of tensorial bodies $P,Q\subset\otimes_{i=1}^{l}\mathbb{R}^{d_{i}},$ $P\sim Q$ if and only if $\delta_{\otimes}^{BM}\left(P,Q\right)=1$ or, equivalently, there exists $T\in GL_\otimes(\otimes_{i=1}^{l}\mathbb{R}^{d_{i}})$ such that $T(P)=Q$.

The set of equivalence classes determined by this relation is denoted by $\mathcal{BM}_{\otimes}\left(\otimes_{i=1}^{l}\mathbb{R}^{d_{i}}\right).$ In  \cite[Theorem 3.13]{tensorialbodies}, it is showed that  $\left(\mathcal{BM}_{\otimes}\left(\otimes_{i=1}^{l}\mathbb{R}^{d_{i}}\right),\log\delta_{\otimes}^{BM}\right)$ is a compact metric space. It is called \textbf{the compactum of tensorial bodies}.

\subsection{Tensorial ellipsoids}
\label{sec:tensorial ellipsoids}
An ellipsoid $\mathcal{E}\subset V$ in a $d$-dimensional vector space is defined as the image of the Euclidean ball $B_2^{d}$ by a linear isomorphism $T:\mathbb{R}^d\rightarrow V.$ In the case of ellipsoids in $\otimes_{i=1}^{l}\mathbb{R}^{d_{i}},$ since we have fixed the scalar product $\langle\cdot,\cdot\rangle_{H},$  we say that $\mathcal{E}\subset\otimes_{i=1}^{l}\mathbb{R}^{d_{i}}$ is an ellipsoid if $\mathcal{E}=T(B_{2}^{d_{1},\ldots,d_{l}})$ for some linear isomorphism $T:\otimes_{i=1}^{l}\mathbb{R}^{d_{i}}\rightarrow\otimes_{i=1}^{l}\mathbb{R}^{d_{i}}.$ 
\medskip

An ellipsoid $\mathcal{E}\subset\otimes_{i=1}^{l}\mathbb{R}^{d_{i}}$ is a \textit{tensorial ellipsoid} if $\mathcal{E}$ is also a tensorial body in $\otimes_{i=1}^{l}\mathbb{R}^{d_{i}}.$ The set of tensorial ellipsoids in $\otimes_{i=1}^{l}\mathbb{R}^{d_{i}}$  is denoted by $\mathscr{E}_{\otimes}(\otimes_{i=1}^{l}\mathbb{R}^{d_{i}}).$ 
The \textit{Hilbertian tensor product of ellipsoids} $\mathcal{E}_i$, $i=1,\ldots,l,$ introduced in \cite{Aubrun2006}, is defined as:
$$
\mathcal{E}_1\otimes_2\cdots\otimes_2\mathcal{E}_l:=T_1\otimes\cdots\otimes T_l(B_{2}^{d_{1},\ldots,d_{l}}),
$$ for $\mathcal{E}_i=T_i(B^{d_{i}}_2).$
It does not depend on the election of the maps $T_i.$ The product $\otimes_2$ gives   examples of tensorial ellipsoids. In this case,  $\mathcal{E}_1\otimes_2\cdots\otimes_2\mathcal{E}_l$ is the closed unit ball of the Hilbert tensor product $\otimes_{H,i=1}^l(\mathbb{R}^{d_{i}},g_{\mathcal{E}_{i}}).$ In particular, the Euclidean ball $B_2^{d_1,\ldots,d_l}$ is such that 
\begin{equation}
\label{eq:euclidean ball is tensorial}
B_2^{d_1,\ldots,d_l}=B_2^{d_1}\otimes_2\cdots\otimes_{2}B_2^{d_l}.
\end{equation}

In \cite[Section 4]{tensorialbodies}, the ellipsoids in the class of tensorial bodies are completely described. It is proved that if $\mathcal{E}$ is a tensorial ellipsoid in $\otimes_{i=1}^{l}\mathbb{R}^{d_{i}}$, then there exist $T_{i}\in GL(d_i),$ $i=1,\ldots,l,$ such that  
\begin{equation}
\label{eq:tensorial ellipsoids char}
\mathcal{E}=T_{1}\otimes\cdots\otimes T_{l}(B_2^{d_1,\ldots,d_l})=T_{1}(B_{2}^{d_{1}})\otimes_2\cdots\otimes_{2}T_{l}(B_{2}^{d_{l}}).
\end{equation}
This shows that tensorial ellipsoids are the image of $B_2^{d_1,\ldots,d_l}$ by elements of $GL_{\otimes}(\otimes_{i=1}^{l}\mathbb{R}^{d_{i}}).$ See \cite[Corollary 4.3]{tensorialbodies}.

\section{The space of tensorial bodies $\mathcal{B}_{\otimes}(\otimes_{i=1}^{l}\mathbb{R}^{d_{i}})$}
\label{sec:hyperspace of convex bodies}

For every pair $P,Q$ of non-empty compact sets contained in a Euclidean space $\mathbb{E}$, the Hausdorff distance $\delta^{H}(P,Q)$ is defined as:
\begin{equation*}
\delta^{H}(P,Q):=\max\left\{ \underset{x\in P}{\sup}\underset{y\in Q}{\inf}\| x-y\|_{\mathbb{E}} ,\underset{y\in Q}{\sup}\underset{x\in P}{\inf}\| y-x\|_{\mathbb{E}} \right\}
\end{equation*}
or, equivalently, by $\delta^{H}(P,Q)=\min\left\{ \lambda\geq0:P\subseteq Q+\lambda B_{\mathbb{E}} ,Q\subseteq P+\lambda B_{\mathbb{E}} \right\}$.
If $P,Q$ are, in addition, $0$-symmetric convex bodies, then we have  the following well known characterization of $\delta^{H},$ see \cite[Theorem 1.8.11]{Schneider1993}:
\begin{equation}
\label{eq:car Hausdorff distance}
\delta^{H}(P,Q)=\underset{x\in\partial B_{\mathbb{E}}} {\sup}\left|g_{P^{\circ}}\left(x\right)-g_{Q^{\circ}}(x)\right|.
\end{equation}

Below, we prove that both the projective $\otimes_\pi$ and the injective $\otimes_\epsilon$ tensor products  of $0$-symmetric convex bodies are continuous functions with respect  to $\delta^{H}$ (Proposition \ref{prop:porducto projectivo e injectivo es continuo con hauss}). 

\begin{lem}
\label{lem:convergencia hausdorff implica convegencia de normas}For
every sequence $\{ C_{n}\}_{n\in\mathbb{N}}\subset\mathcal{B}(d)$ and
$C\in \mathcal{B}\left(d\right),$ we have:
\begin{enumerate}
\item If $g_{C_{n}}\left(\cdot\right)$ converges uniformly on $\partial{B_{2}^{d}}$
to $g_{C}\left(\cdot\right)$, then the same holds for $g_{C_{n}^{\circ}}\left(\cdot\right)$ and $g_{C^{\circ}}\left(\cdot\right).$


\item If $C_{n}$ converges to $C$, in the Hausdorff distance, then $g_{C_{n}}\left(\cdot\right)$
converges uniformly on $\partial{B_{2}^{d}}$ to $g_{C}\left(\cdot\right)$. In particular,
$C_{n}^{\circ}$ goes to $C^{\circ}$ in the Hausdorff distance.
\end{enumerate}
\end{lem}
We do not include the proof of Lemma \ref{lem:convergencia hausdorff implica convegencia de normas}, because it can be directly proved by using (\ref{eq:car Hausdorff distance}). For each $d\in\mathbb{N},$ the function
\begin{align}\label{eq:Lema cotas para continuidad}
\nu:(\mathcal{B}(d),\delta^{H}) & \rightarrow\mathbb{R}\\\nonumber
Q & \mapsto \nu(Q)=\underset{x\in Q}{\sup}\| x\| 
\end{align}
 is uniformly continuous, see \cite[Lemma 4.2]{Antonyan2013}.
%
\begin{lem}
\label{lem:distancia Hausdorff pi}Let $P_{i},Q_{i}\subset\mathbb{R}^{d_{i}},$ $i=1,...,l,$ be $0$-symmetric
convex bodies. Then, for each $i,$ the following holds:  
\[
\delta^{H}(Q_{1}\otimes_{\pi}\cdots\otimes_{\pi}Q_{i}\otimes_{\pi}\cdots\otimes_{\pi}Q_{l},Q_{1}\otimes_{\pi}\cdots\otimes_{\pi}P_{i}\otimes_{\pi}\cdots\otimes_{\pi}Q_{l})\leq\delta^{H}(Q_{i},P_{i})\underset{j\neq i}{\prod}\nu_{j}(Q_{j}).
\]
 Here, $\nu_j:\mathcal{B}(d_j)\rightarrow\mathbb{R}$ is the map of  (\ref{eq:Lema cotas para continuidad}).
\end{lem}

\begin{proof}
Let us fix $i\in\{ 1,...,l\}.$ Observe that from the definition of $\otimes_\pi$ and the properties of the Hausdorff distance \cite[p. 51]{Schneider1993}, it follows that
\begin{gather}
\label{eq:Hausdorff distance of sigmas}\delta^{H}(Q_{1}\otimes_{\pi}\cdots\otimes_{\pi}Q_{i}\otimes_{\pi}\cdots\otimes_{\pi}Q_{l},Q_{1}\otimes_{\pi}\cdots\otimes_{\pi}P_{i}\otimes_{\pi}\cdots\otimes_{\pi}Q_{l})\leq\\
\delta^{H}(\otimes(Q_{1},\ldots,Q_{i},\ldots,Q_{l}),\otimes(Q_{1},\ldots,P_{i},\ldots,Q_{l})).\nonumber
\end{gather}

On the other hand,  if we take $\lambda\geq\delta^{H}(P_i,Q_i)$ then $P_{i}\subseteq Q_{i}+\lambda B_{2}^{d_{i}}$
and $Q_{i}\subseteq P_{i}+\lambda B_{2}^{d_{i}}.$ Thus, for every
$x^{i}\in Q_{i}$, there exists $y^{i}\in P_{i}$ and $u_{i}\in B_{2}^{d_{i}}$  such that $x^{i}=y^{i}+\lambda u_{i}.$ Hence, for every $x^{j}\in Q_{j},$ $j=1,...,i,...,l,$ one has
\begin{align*}
x^{1}\otimes\cdots\otimes x^{i}\otimes\cdots\otimes x^{l} & =x^{1}\otimes\cdots\otimes(y^{i}+\lambda u_{i})\otimes\cdots\otimes x^{l}\\
 & =x^{1}\otimes\cdots\otimes y^{i}\otimes\cdots\otimes x^{l}+x^{1}\otimes\cdots\otimes\lambda u_{i}\otimes\cdots\otimes x^{l}.
\end{align*}
Since $x^{1}\otimes\cdots\otimes y^{i}\otimes\cdots\otimes x^{l}\in\otimes(Q_{1},\ldots,P_{i},\ldots,Q_{l}),$
$x^{j}\in\nu_{j}(Q_{j})B_{2}^{d_{j}}$ for $j\neq i,$ then by (\ref{eq:euclidean ball is tensorial}),
we have
\[
\otimes(Q_{1},\ldots,Q_{i},\ldots,Q_{l})\subseteq\otimes(Q_{1},\ldots,P_{i},\ldots,Q_{l})+\lambda\underset{j\neq i}{\prod}\nu_{j}(Q_{j})B_{2}^{d_{1},...,d_{l}}.
\]

In a similar way, the above inclusion also holds if we exchange $\otimes(Q_{1},\ldots,Q_{i},\ldots,Q_{l})$ and $\otimes(Q_{1},\ldots,P_{i},\ldots,Q_{l}).$
 Therefore,
$$
\delta^{H}(\otimes(Q_{1},\ldots,Q_{i},\ldots,Q_{l}),\otimes(Q_{1},\ldots,P_{i},\ldots,Q_{l}))\leq\lambda\underset{j\neq i}{\prod}\nu_{j}(Q_{j})
$$
and  the result then follows from (\ref{eq:Hausdorff distance of sigmas}). 
\end{proof}

\begin{prop}
\label{prop:porducto projectivo e injectivo es continuo con hauss}The projective $\otimes_\pi$ and the injective $\otimes_\epsilon$ tensor products are continuous functions. That is, for $\alpha=\pi,\epsilon,$ 
\begin{align*}
\otimes_{\alpha}:(\mathcal{B}(d_{1}),\delta^{H})\times\cdots\times(\mathcal{B}(d_{l}),\delta^{H}) & \rightarrow(\mathcal{B}(\otimes_{i=1}^{l}\mathbb{R}^{d_{i}}),\delta^{H})\\
(Q_{1},\ldots,Q_{l}) & \mapsto Q_{1}\otimes_{\alpha}\cdots\otimes_{\alpha}Q_{l}
\end{align*}
is continuous.
\end{prop}

\begin{proof}
First we prove that $\otimes_{\pi}$ is continuous.
To that end for each $i=1,...,l,$ let $\{ Q_{i}^{n}\} _{n\in\mathbb{N}}$
be a sequence in $\mathcal{B}(d_{i})$ converging to $Q_{i}\in\mathcal{B}(d_{i}).$ From the triangle inequality and Lemma
\ref{lem:distancia Hausdorff pi}, we have
\begin{gather*}
\delta^{H}\left(Q_{1}\otimes_{\pi}\cdots\otimes_{\pi}Q_{l},Q^n_{1}\otimes_{\pi}\cdots\otimes_{\pi}Q^n_{l}\right)\leq\\
{\sum}_{i=1}^{l}\delta^{H}\left(Q^n_{1}\otimes_{\pi}\cdots\otimes_{\pi}Q^n_{i-1}\otimes_{\pi}Q_{i}\otimes_{\pi}Q_{i+1}\otimes_{\pi}\cdots\otimes_{\pi}Q_{l},\right.\\
\left. Q^n_{1}\otimes_{\pi}\cdots\otimes_{\pi}Q^n_{i-1}\otimes_{\pi}Q^n_{i}\otimes_{\pi}Q_{i+1}\otimes_{\pi}\cdots\otimes_{\pi}Q_{l}\right)\leq\\
{\sum}_{i=1}^{l}\delta^{H}\left(Q_i,Q^n_i\right)\prod_{j<i}\nu_j\left(Q^n_j\right)\prod_{j>i}\nu_j\left(Q_j\right).
\end{gather*}
Thus, by the continuity of $\nu_{i}$ (\cite[Lemma 4.2]{Antonyan2013})
and the fact that $Q_{i}^{n}$ converges to $Q_{i},$ it follows that $\delta^{H}(Q_{1}\otimes_{\pi}\cdots\otimes_{\pi}Q_{l},Q_{1}^{n}\otimes_{\pi}\cdots\otimes_{\pi}Q_{l}^{n})$ goes to $0.$ This proves the continuity of $\otimes_{\pi}$.

To prove that $\otimes_{\epsilon}$ is continuous, observe
that if $Q_i^n$ and $Q_i$ are as above, then, by Lemma \ref{lem:convergencia hausdorff implica convegencia de normas},
$(Q_{i}^{n})^{\circ}$ converges to $Q_{i}^{\circ}.$ Thus, by the continuity of $\otimes_{\pi},$
$
(Q_{1}^{n})^{\circ}\otimes_{\pi}\cdots\otimes_{\pi}(Q_{l}^{n})^{\circ}$ converges to $Q_{1}^{\circ}\otimes_{\pi}\cdots\otimes_{\pi}Q_{l}^{\circ}.$ Hence, from Lemma \ref{lem:convergencia hausdorff implica convegencia de normas}, 
$((Q_{1}^{n})^{\circ}\otimes_{\pi}\cdots\otimes_{\pi}(Q_{l}^{n})^{\circ})^{\circ}$ converges to $(Q_{1}^{\circ}\otimes_{\pi}\cdots\otimes_{\pi}Q_{l}^{\circ})^{\circ}.$ This shows that
$\delta^{H}(Q_{1}^{n}\otimes_{\epsilon}\cdots\otimes_{\epsilon}Q_{l}^{n},Q_{1}\otimes_{\epsilon}\cdots\otimes_{\epsilon}Q_{l})$ goes to $0,$  and so the continuity of $\otimes_\epsilon$ is proved.
\end{proof}

\subsection{Basic Properties}
\label{ref:basic properties}

From now on, we will refer to $\mathcal{B}_{\otimes}\left(\otimes_{i=1}^{l}\mathbb{R}^{d_{i}}\right)$ as the space of tensorial bodies. Unless otherwise state the topology on it will be the one determined by the Hausdorff distance. In this way, $\mathcal{B}_{\otimes}\left(\otimes_{i=1}^{l}\mathbb{R}^{d_{i}}\right)$ is a topological subspace of the space of $0$-symmetric convex bodies in $\otimes_{i=1}^{l}\mathbb{R}^{d_{i}}.$ Here we will prove that it is a closed and contractible (topological) subspace. See Propositions \ref{prop:Bsigma is closed} and \ref{prop: tensorial bodies are contractil}.

 To shorten notation, when no confusion can arise, we write simply $\mathcal{B}_{\otimes}$ or $\mathcal{B}_{Q_1,\ldots,Q_l}$ instead of  $\mathcal{B}_{\otimes}(\otimes_{i=1}^{l}\mathbb{R}^{d_{i}})$ and $\mathcal{B}_{Q_1,\ldots,Q_l}(\otimes_{i=1}^{l}\mathbb{R}^{d_{i}}).$ 
Recall that for every $0$-symmetric convex body $Q,$ $g_Q$ denotes its Minkowski functional. Also, if $Q\subset\otimes_{i=1}^{l}\mathbb{R}^{d_{i}}$ then $Q^i,$ $i=1,\ldots,l,$ are the $0$-symmetric convex bodies of Remark \ref{rem: canonical section}. 

\begin{prop}
\label{prop:Bsigma is closed}
The following statements hold:
\begin{enumerate}
\item 
Let $Q_i\subset\mathbb{R}^{d_i},$ $i=1,\ldots,l,$ be $0$-symmetric convex bodies. Then  $\mathcal{B}_{Q_{1},...,Q_{l}}\left(\otimes_{i=1}^{l}\mathbb{R}^{d_{i}}\right)$ is a compact convex subset of the space of $0$-symmetric convex bodies in $\otimes_{i=1}^{l}\mathbb{R}^{d_{i}}.$
\item $\mathcal{B}_{\otimes}(\otimes_{i=1}^{l}\mathbb{R}^{d_{i}})$ is closed in the space of $0$-symmetric convex bodies in $\otimes_{i=1}^{l}\mathbb{R}^{d_{i}}.$
\end{enumerate}
\end{prop}

\begin{proof}
(1) The convexity of $\mathcal{B}_{Q_{1},...,Q_{l}}$ follows directly from the properties of the Minkowski sum. In this sense, for any convex body and any $0\leq t\leq 1.$ If $\alpha=\pi,\epsilon$ then $tQ_{1}\otimes_{\alpha}\cdots\otimes_{\alpha}Q_{l}+\left(1-t\right)Q_{1}\otimes_{\alpha}\cdots\otimes_{\alpha}Q_{l}=Q_{1}\otimes_{\alpha}\cdots\otimes_{\alpha}Q_{l}.$  Hence, by (\ref{eq: inclusion convex bodies rc norms}), if $P,Q\in\mathcal{B}_{Q_{1},...,Q_{l}},$ then $tQ+\left(1-t\right)P\in\mathcal{B}_{Q_{1},...,Q_{l}}.$ 
%

To prove that $\mathcal{B}_{Q_{1},...,Q_{l}}$ is compact, it is enough to show that it is closed. The result then follows by the Blaschke selection theorem \cite[Theorem 1.8.6]{Schneider1993} and the fact that $\mathcal{B}_{Q_{1},...,Q_{l}}$ is bounded (every element is contained in $Q_{1}\otimes_{\epsilon}\cdots\otimes_{\epsilon}Q_{l}$).
 
Let $P_{n}\in\mathcal{B}_{Q_{1},...,Q_{l}}$ be a sequence converging to a $0$-symmetric convex body $P$. We will see that $P$ is a tensorial body in $\mathcal{B}_{Q_{1},...,Q_{l}}.$ By (2) of Lemma \ref{lem:convergencia hausdorff implica convegencia de normas}
we have:
\begin{align*}
g_{P}\left(x^{1}\otimes\cdots\otimes x^{l}\right) =\underset{n\rightarrow\infty}{\lim}g_{P_{n}}\left(x^{1}\otimes\cdots\otimes x^{l}\right)&\overset{*}{=}\underset{n\rightarrow\infty}{\lim}g_{Q_{1}}\left(x^{1}\right)\cdots g_{Q_{l}}\left(x^{l}\right)\\
 &=g_{Q_{1}}\left(x^{1}\right)\cdots g_{Q_{l}}\left(x^{l}\right).
\end{align*}
$(*)$ follows from the fact that each $P_n\in\mathcal{B}_{Q_{1},...,Q_{l}},$ see (\ref{eq:rc in gauges}). 
Similar arguments show that 
$
g_{P^{\circ}}\left(x^{1}\otimes\cdots\otimes x^{l}\right)=g_{\left(Q_{1}\right)^{\circ}}\left(x^{1}\right)\cdots g_{\left(Q_{l}\right)^{\circ}}\left(x^{l}\right).
$
Therefore, by \cite[Proposition 3.1]{tensorialbodies}, $P\in\mathcal{B}_{Q_{1},...,Q_{l}}.$

(2)
The closedness of $\mathcal{B}_\otimes$ follows directly from \cite[Proposition 3.7]{tensorialbodies}. To see this, take a sequence $Q_{n}\in\mathcal{B}_\otimes$
converging to a $0$-symmetric convex body $Q.$ By Lemma \ref{lem:convergencia hausdorff implica convegencia de normas}, $g_{Q_n}$ converges uniformly on compact sets to $g_Q.$ Thus, from \cite[Proposition 3.7]{tensorialbodies}, $Q$ must be a tensorial body.

\end{proof}

\begin{prop}
\label{prop: tensorial bodies are contractil}
The space of tensorial bodies $\mathcal{B}_{\otimes}(\otimes_{i=1}^{l}\mathbb{R}^{d_{i}})$ is contractible. 
\end{prop}

\begin{proof}
We will define a contracting homotopy from $\mathcal{B}_{\otimes}$ to $B_2^{d_1}\otimes_{\pi}\cdots\otimes_{\pi}B_2^{d_l}.$
For every tensorial body $Q\subset\otimes_{i=1}^{l}\mathbb{R}^{d_{i}}$ and every real number $t\in[0,1],$ let  $H$ be defined as:
$$H(Q,t)=(1-2t)Q+2t(Q^1\otimes_\pi\cdots\otimes_\pi Q^l)\text{ for }0\leq t\leq\frac{1}{2},\text{ and}$$
 $$H(Q,t)=((2-2t)Q^1+(2t-1)B_2^{d_1})\otimes_\pi\cdots\otimes_\pi((2-2t)Q^l+(2t-1)B_2^{d_l}),$$
 for $\frac{1}{2}\leq t\leq 1$. 
Below, we show that $H:\mathcal{B}_{\otimes}\times[0,1]\rightarrow\mathcal{B}_{\otimes}$ is the desired homotopy.
 
To prove that it is well defined, observe that since every tensorial body $Q$ belongs to its corresponding $\mathcal{B}_{Q^1,\ldots,Q^l},$ then, by the convexity of $\mathcal{B}_{Q^1,\ldots,Q^l}$ (Proposition \ref{prop:Bsigma is closed}), it holds that $(1-2t)Q+2t(Q^1\otimes_\pi\cdots\otimes_\pi Q^l)$ is a tensorial body for all $t\in\left[0,\frac{1}{2}\right].$ Also,  by definition of $H,$ $H(Q,t)$ is a tensorial body for $\frac{1}{2}\leq t\leq 1$.
Clearly $H(Q,0)$ is the identity map on $\mathcal{B}_{\otimes}$ and $H(Q,1)=B_2^{d_1}\otimes_{\pi}\cdots\otimes_{\pi}B_2^{d_l}.$ 

It remains to prove the continuity of $H$. To this end, 
let $t_n$ and  $Q_n$ be sequences of  both real numbers in $\left[0,\frac{1}{2}\right]$ and of tensorial bodies converging to $t$ and $Q$ respectively. From Lemma \ref{lem:convergencia hausdorff implica convegencia de normas} and \cite[Proposition 3.7]{tensorialbodies}, $g_{Q_n^{i}}$ converges uniformly (on compact sets) to $g_{Q^i},$ for each $i=1,\ldots,l.$  This, along with  (\ref{eq:car Hausdorff distance}) and (1) in Lemma \ref{lem:convergencia hausdorff implica convegencia de normas}, proves that $Q_n^{i}$ goes to $Q^i,$ for each $i$. Hence, by the continuity of $\otimes_\pi$, $Q_n^{1}\otimes_\pi\cdots\otimes_\pi Q_n^l$ converges to $Q^{1}\otimes_\pi\cdots\otimes_\pi Q^l$ in the Hausdorff distance. Now, by the triangle inequality, we have:
\begin{gather*}
\delta^{H}(H(Q_n,t_n),H(Q,t))\leq\delta^{H}(H(Q_n,t_n),H(Q,t_n))+\delta^{H}(H(Q,t_n),H(Q,t))\\\leq (1-2t_n)\delta^{H}(Q_n,Q)+2t_n\delta^{H}(Q_n^{1}\otimes_\pi\cdots\otimes_\pi Q_n^l,Q^{1}\otimes_\pi\cdots\otimes_\pi Q^l)\\+ 2|t_n-t|(\nu(Q)+\nu(Q^{1}\otimes_\pi\cdots\otimes_\pi Q^l)),
\end{gather*}
where $\nu$ is the map defined in (\ref{eq:Lema cotas para continuidad}). The convergence of $H(Q_n,t_n)$ to $H(Q,t)$ then follows from the last inequality and the previous discussion. In the same manner,  the continuity on $\left[\frac{1}{2},1\right]$ follows from the continuity of both the projective tensor product $\otimes_\pi$ and the map $t\mapsto\ (2-2t)Q^i+(2t-1)B_2^{d_i}$, $i=1,...,l.$ 
\end{proof}

We finish this section with  two results that will be used  in the proof of Theorem \ref{thm:the action is proper}.

\begin{lem}
\label{lem:LEMA natalia}
Let $\varepsilon>0$ and let $P\subset\otimes_{i=1}^{l}\mathbb{R}^{d_{i}}$ be a tensorial body such that $2\varepsilon B_{2}^{d_{1},\ldots,d_{l}}\subseteq P$. If $\delta^{H}\left(P,Q\right)<\varepsilon,$ for some tensorial body $Q$, then $\varepsilon B_{2}^{d_{1},\ldots,d_{l}}\subseteq Q$. 
\end{lem}


We do not include the proof of Lemma \ref{lem:LEMA natalia} because it follows from \cite[Lemma 3.1]{Antonyan2013}  by making $x_0$ equals to zero. 

\begin{lem}
\label{lem:lema rel compactos}
Let $\varepsilon>0$ and let $P\subset\otimes_{i=1}^{l}\mathbb{R}^{d_{i}}$ be a tensorial body such that  $2\varepsilon B_{2}^{d_{1},\ldots,d_{l}}\subseteq P$. Then the set, $V_{P}\left(\varepsilon\right):=\left\{ Q\in\mathcal{B}_{\otimes}(\otimes_{i=1}^{l}\mathbb{R}^{d_{i}}):\delta^{H}(P,Q)<\varepsilon\right\}$  is a relatively compact set in the space of tensorial bodies.
\end{lem}

\begin{proof}
We will prove that each sequence in $V_{P}\left(\varepsilon\right)$ has a convergent subsequence in $\mathcal{B}_\otimes.$ Suppose that $2\varepsilon B_{2}^{d_{1},\ldots,d_{l}}\subseteq P$
and  let $Q_{k},$ $k\in\mathbb{N},$ be a sequence contained in $V_{P}\left(\varepsilon\right).$
Clearly $Q_{k},$ $k\in\mathbb{N},$ is a bounded sequence of non-empty compact convex sets.
Thus, from the Blaschke selection theorem (\cite[Theorem 1.8.6]{Schneider1993}) there exists a subsequence $Q_{k_i}$ such that $Q_{k_i}$ converges to a non-empty compact convex set $Q\subset\otimes_{i=1}^l\mathbb{R}^{d_{i}}$. By Lemma \ref{lem:LEMA natalia},
$\varepsilon B_{2}^{d_{1},\ldots,d_{l}}\subseteq Q_{k_{i}}$ for all
$i\in\mathbb{N},$ therefore $\varepsilon B_{2}^{d_{1},\ldots,d_{l}}\subseteq Q$ and $Q$ is a $0$-symmetric convex body. Hence, from Proposition \ref{prop:Bsigma is closed}, it follows that $Q\in\mathcal{B}_{\otimes}.$ This completes the proof.
\end{proof}

\section{A natural action on the space of tensorial bodies $\mathcal{B}_{\otimes}(\otimes_{i=1}^l\mathbb{R}^{d_i})$}
\label{sec:the action of GLtensor}
Recall that $GL_\otimes(\otimes_{i=1}^l\mathbb{R}^{d_i})$ is the set of linear isomorphisms on $\otimes_{i=1}^l\mathbb{R}^d_i$ that preseve decomposable vectors (see Section 2). In a similar way,  $O_\otimes(\otimes_{i=1}^l\mathbb{R}^{d_i})$  consists of the orthogonal maps on $\otimes_{i=1}^l\mathbb{R}^{d_i}$ which enjoy of the same property. That is, $$O_\otimes(\otimes_{i=1}^l\mathbb{R}^{d_i}):=GL_\otimes(\otimes_{i=1}^l\mathbb{R}^{d_i})\cap O(\otimes_{i=1}^l\mathbb{R}^{d_i}).$$ 


In \cite[Proposition 3.11 and Theorem 3.12]{tensorialbodies}, it is proved that $GL_\otimes(\otimes_{i=1}^l\mathbb{R}^{d_i})$ is a closed subgroup of $GL(\otimes_{i=1}^l\mathbb{R}^{d_i})$ which preserves tensorial bodies, see also (\ref{eq: glotimes preserva cuerpos tensoriales}). This allows us to define a natural action of $GL_\otimes(\otimes_{i=1}^l\mathbb{R}^{d_i})$ on the space $\mathcal{B}_{\otimes}(\otimes_{i=1}^l\mathbb{R}^{d_i}).$ It is given by:
\begin{alignat*}{1}
GL_\otimes(\otimes_{i=1}^l\mathbb{R}^{d_i})\times\mathcal{B}_{\otimes}(\otimes_{i=1}^l\mathbb{R}^{d_i}) & \longrightarrow\mathcal{B}_{\otimes}(\otimes_{i=1}^l\mathbb{R}^{d_i})\\
\left(T,Q\right) & \mapsto TQ:=\left\{ Tu:u\in Q\right\} .
\end{alignat*}
Clearly, it is a continuous action. Indeed, it is the restriction of the natural action of the general linear group $GL(\otimes_{i=1}^l\mathbb{R}^{d_i})$ on the space of $0$-symmetric convex bodies in $\otimes_{i=1}^l\mathbb{R}^{d_i},$ which is continuous (\cite[p. 210]{banacamazurcompactum}).  
\medskip

For simplicity of notation, we usually write $GL_\otimes$ and $O_\otimes$ instead of $GL_\otimes(\otimes_{i=1}^l\mathbb{R}^{d_i})$ and $O_\otimes(\otimes_{i=1}^l\mathbb{R}^{d_i}).$
Remember that throughout the paper $l\geq2$ and $d_{i}\geq2$ are integers. 
\begin{thm}
\label{thm:the action is proper}
The action of
$GL_{\otimes}(\otimes_{i=1}^{l}\mathbb{R}^{d_{i}})$ on
$\mathcal{B}_{\otimes}(\otimes_{i=1}^{l}\mathbb{R}^{d_{i}})$
is proper.
\end{thm}

\begin{proof}
First notice that since $GL_\otimes$ is a closed subgroup of $GL(\otimes_{i=1}^{l}\mathbb{R}^{d_{i}})$, see \cite[Proposition 3.11]{tensorialbodies}, it is a locally compact Lie group.

Now, let $P\in\mathcal{B}_\otimes$
and $\varepsilon>0$ be such that $2\varepsilon B_{2}^{d_{1},\ldots,d_{l}}\subseteq P.$
We claim that $V_{P}(\varepsilon)$ is a small neighborhood of $P.$ 
To prove this, observe that for each $0$-symmetric convex body $C\subset\otimes_{i=1}^{l}\mathbb{R}^{d_{i}}$, there exists $\lambda>0$ such that $\lambda B_{2}^{d_{1},\ldots,d_{l}}\subseteq C.$
We will show that the transporter 
$
\Gamma=\left\{ T\in GL_{\otimes}:TV_{P}(\varepsilon)\cap V_{C}(\lambda)\neq\emptyset\right\} 
$
has compact closure in $GL_{\otimes}.$ To that end, we will prove that every sequence in $\Gamma$ has a convergent subsequence in $GL_\otimes.$

Let $T_{k}$, $k\in\mathbb{N},$ be a sequence contained in $\Gamma.$ 
Then, for every $k\in\mathbb{N},$ there exists $Q_{k}\in V_{P}(\varepsilon)$
such that $T_{k}(Q_{k})\in V_{C}(\lambda).$ Thus, by Lemma \ref{lem:LEMA natalia}, $\varepsilon B_{2}^{d_{1},\ldots,d_{l}}\subseteq Q_{k}$
for all $k.$ Also, since $\delta^{H}(T_{k}(Q_{k}),C)<\lambda,$
we have that $T_{k}(Q_{k})\subseteq C+\lambda B_{2}^{d_{1},\ldots,d_{l}}.$
Hence, 
\[
T_{k}(\varepsilon B_{2}^{d_{1},\ldots,d_{l}})\subseteq T_{k}(Q_{k})\subseteq C+\lambda B_{2}^{d_{1},\ldots,d_{l}}.
\]
The latter implies that the sequence $\| T_{k}\|$ is bounded
(here $\|T_{k}\| $ is the operator
norm of $T_{k}$ on $\otimes_{H,i=1}^{l}\mathbb{R}^{d_{i}}$). We thus have that 
$T_{k}$ is a bounded sequence of operators between finite dimensional spaces. Therefore $T_{k}$ must have a convergent subsequence.

Let us denote by $T\in\mathcal{L}(\otimes_{H,i=1}^{l}\mathbb{R}^{d_{i}})$
 to the limit of such subsequence $T_{k_{i}}$. We will prove that $T$ belongs to $GL_\otimes.$ Since $GL_\otimes$ is a closed subgroup of $GL(\otimes_{i=1}^l\mathbb{R}^{d_i})$, see \cite[Proposition 3.11]{tensorialbodies},
it is enough to prove that $T$ is a linear isomorphism. To do this,
observe that $Q_{k_{i}}\in V_{P}(\varepsilon)$
and $T_{k_{i}}(Q_{k_{i}})\in V_{C}(\lambda).$
Thus, from Lemma \ref{lem:lema rel compactos}, there exists a sub-subsequence
$Q_{k_{i_{j}}}$ such that $Q_{k_{i_{j}}}$
and $T_{k_{i_j}}\left(Q_{k_{i_j}}\right)$ converge to some tensorial bodies $Q,$ $D$ respectively. 
This yields to $TQ=D$ which proves that $T$ is a linear isomorphism. This completes
the proof.
\end{proof}

Let us recall that  $\mathscr{E}_{\otimes}(\otimes_{i=1}^{l}\mathbb{R}^{d_{i}})$ consists of the ellipsoids in the class of tensorial bodies, see Section 2.1. 
The next result extends, to the class of tensorial bodies,  Corollary 3.10 of \cite{Antonyan2013} which proves that the space of ellipsoids in $\mathbb{R}^d$ is homeomorphic to $\mathbb{R}^p,$ with $p=\frac{d(d+1)}{2}.$ 

\begin{cor}\label{cor:glsigma over osigma and ellipsoids}
$\mathscr{E}_{\otimes}(\otimes_{i=1}^{l}\mathbb{R}^{d_{i}})$ is homeomorphic to $\mathbb{R}^p$ with $p=\frac{d_1(d_1+1)}{2}+\cdots+\frac{d_l(d_l+1)}{2}.$
\end{cor}

\begin{proof}
 Observe that $O_\otimes$ is the estabilizer of $B_2^{d_1,...,d_l}$ thus, by Theorem \ref{thm:the action is proper} and \cite[Proposition 1.1.5]{Palais2}, the $GL_\otimes$-orbit of  $B_2^{d_1,...,d_l}$ is homeomorphic to the quotient $GL_\otimes/O_\otimes$. The result then follows from \cite[Corollary 4.3]{tensorialbodies}, which shows that the $GL_\otimes$-orbit of $B_2^{d_1,...,d_l}$ is $\mathscr{E}_{\otimes}\left(\otimes_{i=1}^{l}\mathbb{R}^{d_{i}}\right),$ and from Corollary \ref{cor: structure GL/O tensor}.
\end{proof}

\subsubsection*{The $GL_\otimes$-equivariant retraction $conv_\otimes$}

Given a tensorial body $Q\subset\otimes_{i=1}^{l}\mathbb{R}^{d_{i}},$ we define $conv_\otimes$ as the map sending $Q$ to the projective tensor product of its associated $0$-symmetric bodies $Q^{i},$ $i=1,\ldots,l,$ see Remark \ref{rem: canonical section}. That is,  
\begin{align*}
conv_{\otimes}:\mathcal{B}_{\otimes}\left(\otimes_{i=1}^{l}\mathbb{R}^{d_{i}}\right) & \longrightarrow\mathcal{B}_{\otimes}\left(\otimes_{i=1}^{l}\mathbb{R}^{d_{i}}\right)\\
Q & \mapsto Q^{1}\otimes_{\pi}\cdots\otimes_{\pi}Q^{l}.
\end{align*}
The set of fixed points of $conv_\otimes$ is  $\Pi:=\{Q\in\mathcal{B}_{\otimes}(\otimes_{i=1}^{l}\mathbb{R}^{d_{i}}):conv_\otimes(Q)=Q\}$. Below, we prove that $conv_\otimes$ is a $GL_\otimes$-equivariant retraction from the space of tensorial bodies onto $\Pi,$ see (3) in Proposition \ref{prop:conv function}. 
\begin{lem}
 $\Pi$ is the image of $\mathcal{B}(d_1)\times\cdots\times\mathcal{B}(d_l)$ under the projective tensor product $\otimes_\pi.$
\end{lem}
\begin{proof}
Let us suppose that $Q\in\Pi,$ then $Q=Q^1\otimes_\pi\cdots\otimes_\pi Q^l$ and it is the image of  $(Q^1,\ldots,Q^l)$ under $\otimes_\pi.$ On the other hand, if $Q$ is such that $Q=Q_1\otimes_\pi\cdots\otimes_\pi Q_l$ for some $0$-symmetric convex bodies $Q_i\subset\mathbb{R}^{d_i},$ then $Q\in\mathcal{B}_{Q_1,\ldots,Q_l}\cap\mathcal{B}_{Q^1,\ldots,Q^l}.$ Therefore, by \cite[Proposition 3.6]{tensorialbodies}, there exist $\lambda_i>0,$ $i=1,\ldots,l,$ such that $\lambda_1\cdots\lambda_l=1$ and $Q^i=\lambda_iQ_i,$ for all $i.$ From this, 
$Q^1\otimes_\pi\cdots\otimes_\pi Q^l=Q_1\otimes_\pi\cdots\otimes_\pi Q_l$ 
and $Q\in\Pi$, so the proof is completed. 
\end{proof}

The next proposition collects the basic properties of $conv_\otimes$. It is based on an alternative definition of this map stated in terms of the set of decomposable vectors. We follow the notation of \cite{maite}. There, this set is denoted by $$\Sigma_{\mathbb{R}^{d_1},\ldots,\mathbb{R}^{d_l}}:=\{x^1\otimes\cdots\otimes x^l:x^i\in\mathbb{R}^{d_i},i=1,\ldots,l\}.$$

\begin{prop}
\label{prop:conv function}
The following statements hold:
\begin{enumerate}
\item For every tensorial body $Q\subset\otimes_{i=1}^{l}\mathbb{R}^{d_{i}},$ $conv_{\otimes}(Q)=conv(Q\cap\Sigma_{\mathbb{R}^{d_1},\ldots,\mathbb{R}^{d_l}}).$ 

\item The map $conv_{\otimes}$ is constant on every $\mathcal{B}_{Q_1,\ldots,Q_l}(\otimes_{i=1}^{l}\mathbb{R}^{d_{i}}).$


\item The map $conv_{\otimes}$ is a $GL_\otimes$-equivariant retraction from $\mathcal{B}_{\otimes}(\otimes_{i=1}^{l}\mathbb{R}^{d_{i}})$ onto $\Pi$.
\end{enumerate}
\end{prop}

\begin{proof}
To prove (1) observe that from Theorem \ref{thm:equivalence of tensor body}, each tensorial body $Q$ is a tensorial body with respect to $Q^1,\ldots,Q^l$, so by (\ref{eq:rc in gauges}), we have $Q^{1}\otimes_{\pi}\cdots\otimes_{\pi}Q^{l}\subseteq conv(Q\cap\Sigma_{\mathbb{R}^{d_1},\ldots,\mathbb{R}^{d_l}})$. 
To show the other inclusion, notice that every $w=x^1\otimes\cdots\otimes x^l\in Q$ can be written as $w=y^1\otimes\cdots\otimes y^l$ where $y^{i}=g_{Q^i}(\frac{x^i}{g_{Q^i}(x^i)})x^i,$ $i=1,\ldots,l-1,$ and $y^l=g_{Q^1}(x^1)\cdots g_{Q^{l-1}}(x^{l-1})x^l.$ Clearly each $y^i\in Q^i$ for $i=1,\ldots,l-1$ and, from (\ref{eq:rc in gauges}), $y^l\in Q^l.$ Hence, $w\in Q^{1}\otimes_{\pi}\cdots\otimes_{\pi}Q^{l}.$ The desired inclusion then follows by convexity.
\medskip

(2) We will prove that for every $Q\in\mathcal{B}_{Q_1,\ldots,Q_l},$
\begin{equation}
\label{eq:conv is constant on Bq1q2}
conv_\otimes(Q)=Q_{1}\otimes_{\pi}\cdots\otimes_{\pi}Q_{l}.
\end{equation}
Let $Q\in\mathcal{B}_{Q_1,\ldots,Q_l},$ then $Q\in\mathcal{B}_{Q_1,\ldots,Q_l}\cap\mathcal{B}_{Q^1,\ldots,Q^l}$ and, from \cite[Proposition 3.6]{tensorialbodies},  there exist $\lambda_i>0,$ $i=1,\ldots,l,$ such that $\lambda_1\cdots\lambda_l=1$ and $Q^i=\lambda_iQ_i,$ for all $i.$ Hence, $Q^1\otimes_\pi\cdots\otimes_\pi Q^l=Q_1\otimes_\pi\cdots\otimes_\pi Q_l$ and $conv_{\otimes}(Q)=Q_1\otimes_\pi\cdots\otimes_\pi Q_l,$ as desired. 

(3) 
To prove the equivariance of $conv_\otimes,$ it is enough to recall that linear maps preserve convex hulls. Hence, $T(conv(Q\cap\Sigma_{\mathbb{R}^{d_1},\ldots,\mathbb{R}^{d_l}}))=conv(T(Q\cap\Sigma_{\mathbb{R}^{d_1},\ldots,\mathbb{R}^{d_l}}))=conv((TQ)\cap\Sigma_{\mathbb{R}^{d_1},\ldots,\mathbb{R}^{d_l}}),$ for every $T\in GL_\otimes.$ The latter along with (1) shows that $T(conv_\otimes(Q))=conv_\otimes(TQ).$

 To prove that $conv_\otimes$ is continuous, let $Q_{k},$ $k\in\mathbb{N},$ be a sequence in
$\mathcal{B}_{\otimes}$
converging to a tensorial body $Q.$ We will see that $conv_\otimes(Q_k)$ converges to $conv_\otimes(Q).$ By (2) in Lemma \ref{lem:convergencia hausdorff implica convegencia de normas}, $g_{Q_k}$ converges uniformly (on compact sets) to $g_Q.$ Thus, from \cite[Proposition 3.7]{tensorialbodies}, $g_{(Q_{k}^i)^\circ}$ converges uniformly (on compact sets) to $g_{(Q^i)^{\circ}}.$ So, by (\ref{eq:car Hausdorff distance}), $Q_{k}^{i}$ goes to $Q^i$ for each $i=1,\ldots,l.$ Hence, by the continuity of $\otimes_\pi$ (Proposition \ref{prop:porducto projectivo e injectivo es continuo con hauss}), $Q_{k}^1\otimes_\pi\cdots\otimes_\pi Q_{k}^l$ converges to $Q^{1}\otimes_\pi\cdots\otimes_\pi Q^{l}$ and   $conv_\otimes$ is continuous.

To show that $conv_\otimes$ is a retraction onto $\Pi,$ notice that each tensorial body  $Q$ is such that $Q, Q^{1}\otimes_{\pi}\cdots\otimes_{\pi}Q^{l}\in\mathcal{B}_{Q^1,\ldots,Q^l},$ then by (\ref{eq:conv is constant on Bq1q2}), $conv_\otimes(conv_\otimes(Q))=conv_\otimes(Q^{1}\otimes_{\pi}\cdots\otimes_{\pi}Q^{l})=Q^{1}\otimes_{\pi}\cdots\otimes_{\pi}Q^{l}=conv_\otimes(Q).$
\end{proof}

\subsection{A global slice for the space of tensorial bodies $\mathcal{B}_{\otimes}(\otimes_{i=1}^{l}\mathbb{R}^{d_i})$}
\label{sec:global slice}

Given a $0$-symmetric convex body $Q\subset\mathbb{R}^d$, the L\"{o}wner ellipsoid $L\ddot{o}w(Q)$ and John ellipsoid $John(Q)$ are defined as the ellipsoid of minimal  volume (resp. maximal volume) containing $Q$ (resp. contained in $Q$).  Both ellipsoids are fundamental tools in the study of convex bodies and finite dimensional normed spaces (see for instance \cite[Chapter 3]{Tomczak-Jaegermann1989}). 
One of their main features is the uniqueness, which was proved by F. John \cite{JohnF}. This property enables the definition of the L\"{o}wner map, $L\ddot{o}w,$ as the map sending each $0$-symmetric convex body $Q\subset\mathbb{R}^d$ to its L\"{o}wner ellipsoid  $L\ddot{o}w(Q),$ see 
\cite[Section 3.2]{Antonyan2013}. Among its fundamental properties, in \cite[Theorem 3.6]{Antonyan2013}, it is proved that the L\"{o}wner map is a $GL(d)$-equivariant retraction onto the set of ellipsoids in $\mathbb{R}^d$. By means of this  map, we will exhibit a compact $O_\otimes$-global slice for the space of tensorial bodies (Theorem \ref{thm:compact global slice}).
\medskip

The relation between L\"{o}wner and John ellipsoids and the tensor products $\otimes_\pi,$ $\otimes_\epsilon$ is established in  \cite[Lemma 1]{Aubrun2006} and \cite[Proposition 3.14]{MaiteLuisa2}. There, it is proved that L\"{o}wner and John ellipsoids are preserved under the projective $\otimes_\pi$ and the injective tensor product $\otimes_\epsilon,$ respectively. 
In \cite[Lemma 1]{Aubrun2006}, it is showed  that for every tuple of $0$-symmetric convex bodies $Q_i\subset\mathbb{R}^{d_i},$ $i=1,\ldots,l,$ the L\"{o}wner ellipsoid of the projective tensor product is the Hilbertian tensor product of the  L\"{o}wner ellipsoids $L\ddot{o}w(Q_i).$ That is,
 \begin{equation}
 \label{eq:low ellipsoid and pprojective}
 L\ddot{o}w(Q_1\otimes_\pi\cdots\otimes_\pi Q_l)=L\ddot{o}w(Q_1)\otimes_2\cdots\otimes_2 L\ddot{o}w(Q_l).
 \end{equation}



As a consequence of (\ref{eq:low ellipsoid and pprojective}), L\"{o}wner ellipsoids of projective tensor products of $0$-symmetric convex bodies are tensorial ellipsoids. This allows us to define the map $l_{\otimes}$ as the composition of two retractions $l_{\otimes}:=L\ddot{o}w\circ conv_{\otimes},$ 
\begin{align*}
l_{\otimes}:\mathcal{B}_{\otimes}(\otimes_{i=1}^{l}\mathbb{R}^{d_{i}}) & \rightarrow\mathscr{E}_{\otimes}(\otimes_{i=1}^{l}\mathbb{R}^{d_{i}})\\
Q & \mapsto L\ddot{o}w(Q^{1}\otimes_{\pi}\cdots\otimes_{\pi}Q^{l}).
\end{align*}

Below we prove that $l_\otimes$ is a $GL_\otimes$-equivariant retraction onto the space of tensorial ellipsoids (Proposition \ref{prop:lsigma es RETRACCION}). This map will be used to show that the subset $\mathscr{L}_{\otimes}(\otimes_{i=1}^{l}\mathbb{R}^{d_{i}})$ which consists of the tensorial bodies $Q\subset\otimes_{i=1}^{l}\mathbb{R}^{d_{i}}$ such that $l_{\otimes}(Q)=B_{2}^{d_{1},\ldots,d_{l}}$ is a compact global $O_\otimes$-slice for the space of tensorial bodies $\mathcal{B}_\otimes(\otimes_{i=1}^{l}\mathbb{R}^{d_{i}}),$ see Theorem \ref{thm:compact global slice}.
To simplify notation, we will write $\mathscr{E}_\otimes$ and $\mathscr{L}_\otimes$ instead of $\mathscr{E}_{\otimes}(\otimes_{i=1}^{l}\mathbb{R}^{d_{i}})$ and $\mathscr{L}_{\otimes}(\otimes_{i=1}^{l}\mathbb{R}^{d_{i}}).$  

\begin{prop}
\label{prop:lsigma es RETRACCION} The map $l_{\otimes}$ is a $GL_{\otimes}$-equivariant retraction from
$\mathcal{B}_{\otimes}\left(\otimes_{i=1}^{l}\mathbb{R}^{d_{i}}\right)$
onto $\mathscr{E}_{\otimes}\left(\otimes_{i=1}^{l}\mathbb{R}^{d_{i}}\right).$ 
\end{prop}
\begin{proof}
The map $l_{\otimes}$ is continuous and $GL_{\otimes}$-equivariant due to  the L\"{o}wner map $L\ddot{o}w$ is a continuous $GL(\otimes_{i=1}^{l}\mathbb{R}^{d_{i}})$-equivariant  map \cite[Theorem 3.6]{Antonyan2013}, and the map $conv_\otimes$ is continuous and $GL_\otimes$-equivariant  (see Proposition \ref{prop:conv function}).
To prove that $l_{\otimes}$ is a retraction. Let $\mathcal{E}\subset\otimes_{i=1}^{l}\mathbb{R}^{d_{i}}$ be a tensorial ellipsoid,
then, by \cite[Corollary 4.3]{tensorialbodies}, there exist $T_{i}\in GL(d_{i}),$ $i=1,...,l,$ such that $\mathcal{E}=T_{1}\otimes\cdots\otimes T_{l}(B_{2}^{d_{1},\ldots,d_{l}}).$
Thus,
\begin{gather*}
l_{\otimes}(\mathcal{E}) =l_{\otimes}(T_{1}\otimes\cdots\otimes T_{l}(B_{2}^{d_{1},\ldots,d_{l}}))
 =(T_{1}\otimes\cdots\otimes T_{l})l_{\otimes}(B_{2}^{d_{1},\ldots,d_{l}})\\
 =T_{1}\otimes\cdots\otimes T_{l}(B_{2}^{d_{1},\ldots,d_{l}})
  =\mathcal{E}.
\end{gather*}
 The third equality follows from (\ref{eq:euclidean ball is tensorial}), (\ref{eq:conv is constant on Bq1q2})  and (\ref{eq:low ellipsoid and pprojective}).
\end{proof}


\begin{prop}
\label{prop:-Lsigma es compacto} $\mathscr{L}_{\otimes}(\otimes_{i=1}^{l}\mathbb{R}^{d_{i}})$ enjoys of the following properties:
\begin{enumerate}
\item $\mathscr{L}_{\otimes}(\otimes_{i=1}^{l}\mathbb{R}^{d_{i}})$
is $O_{\otimes}$-invariant.

\item The $GL_{\otimes}$-saturation of $\mathscr{L}_{\otimes}(\otimes_{i=1}^{l}\mathbb{R}^{d_{i}})$
coincides with $\mathcal{B}_{\otimes}(\otimes_{i=1}^{l}\mathbb{R}^{d_{i}}).$

\item Let $T\in GL_{\otimes}(\otimes_{i=1}^{l}\mathbb{R}^{d_{i}}).$
If $T\left(\mathcal{\mathscr{L}}_{\otimes}(\otimes_{i=1}^{l}\mathbb{R}^{d_{i}})\right)\cap\mathscr{L}_{\otimes}(\otimes_{i=1}^{l}\mathbb{R}^{d_{i}})\neq\emptyset,$
then 
$T$ is an orthogonal map.

\item $\mathscr{L}_{\otimes}(\otimes_{i=1}^{l}\mathbb{R}^{d_{i}})$
is compact.
\end{enumerate}
\end{prop}
\begin{proof}
(1). Let $Q\in\mathscr{L}_{\otimes}$
and let $U$ be an orthogonal map in $O_{\otimes}.$
By the equivariance of $l_\otimes$ (Proposition \ref{prop:lsigma es RETRACCION}),
$l_{\otimes}\left(UQ\right)=U(l_{\otimes}(Q))=U(B_{2}^{d_{1},\ldots,d_{l}})=B_{2}^{d_{1},\ldots,d_{l}}.$
Therefore $U(Q)\in\mathscr{L}_{\otimes}$ and the invariance of $\mathscr{L}_{\otimes}$ is proved.

(2). Let $Q\subset\otimes_{i=1}^{l}\mathbb{R}^{d_{i}}$ be an arbitrary tensorial body, we will see that it belongs to the $GL_\otimes$-saturation of $\mathscr{L}_{\otimes}$. Since for each $i=1,\ldots,l,$ there exists $T_{i}\in GL(d_i)$ such that
$L\ddot{o}w(Q^{i})=T_{i}(B_{2}^{d_{i}}),$ then,
by setting $P:=T_{1}^{-1}\otimes\cdots\otimes T_{l}^{-1}(Q),$ we have:
\begin{gather*}
l_{\otimes}\left(P\right) =l_{\otimes}\left(T_{1}^{-1}\otimes\cdots\otimes T_{l}^{-1}\left(Q\right)\right)
 \overset{*}{=}T_{1}^{-1}\otimes\cdots\otimes T_{l}^{-1}\left(l_{\otimes}\left(Q\right)\right)\\
 \overset{**}{=}T_{1}^{-1}\otimes\cdots\otimes T_{l}^{-1}\left(L\ddot{o}w(Q^{1}) \otimes_{2}\cdots\otimes_{2}L\ddot{o}w(Q^{l})\right)\\
=T_{1}^{-1}\otimes\cdots\otimes T_{l}^{-1}\left(T_{1}(B_{2}^{d_{1}})\otimes_{2}\cdots\otimes_{2}T_{l}(B_{2}^{d_{l}})\right)
  \\ \overset{***}{=}B_{2}^{d_{1}}\otimes_{2}\cdots\otimes_{2}B_{2}^{d_{l}}=
  B_{2}^{d_1,\ldots,d_l}.
\end{gather*}
(*) follows from Proposition \ref{prop:lsigma es RETRACCION}. (**) follows from (\ref{eq:low ellipsoid and pprojective}). (***) follows from the second equality in (\ref{eq:tensorial ellipsoids char}).

Therefore $P\in\mathscr{L}_{\otimes}$ and $Q=T_{1}\otimes\cdots\otimes T_{l}(P)$ belongs to the $GL_{\otimes}$-saturation of $\mathscr{L}_{\otimes}.$

(3).
Let $T\in Gl_{\otimes}$
such that $TQ\in\mathcal{\mathscr{L}}_\otimes,$ for some $Q\in\mathcal{\mathscr{L}}_{\otimes}.$
Then, by the $GL_\otimes$-equivariance of $l_\otimes,$ $T(B_2^{d_1\ldots,d_l})=T(l_{\otimes}(Q))=l_{\otimes}(TQ)=B_2^{d_1\ldots,d_l}.$ This proves that  $T$ is an orthogonal map in $GL_\otimes$. So, $T\in O_\otimes.$

(4). We will prove that any sequence in $\mathscr{L}_{\otimes}$ has a convergent subsequence.
Let $Q_{k}$ be a sequence contained
in $\mathscr{L}_{\otimes}.$
Then for each $k,$ 
$
L\ddot{o}w(Q_{k}^{1}\otimes_{\pi}\cdots\otimes_{\pi}Q_{k}^{l})=B_2^{d_1,\ldots,d_l}
$
and so $Q_{k}^{1}\otimes_{\pi}\cdots\otimes_{\pi}Q_{k}^{l}$ belongs to the compact set  $\mathscr{L}(\otimes_{i=1}^{l}\mathbb{R}^{d_{i}})\subset\mathcal{B}(\otimes_{i=1}^{l}\mathbb{R}^{d_{i}})$ of \cite[Remark 1]{banacamazurcompactum}. Consequently,
there exists a subsequence $Q_{k_{j}}^{1}\otimes_{\pi}\cdots\otimes_{\pi}Q_{k_{j}}^{l}$
converging to some $0$-symmetric convex body $D\subset\otimes_{i=1}^{l}\mathbb{R}^{d_{i}},$
such that $L\ddot{o}w(D)=B_2^{d_1,\ldots,d_l}.$ Hence, from (2) in Lemma \ref{lem:convergencia hausdorff implica convegencia de normas} and \cite[Proposition 3.7]{tensorialbodies}, we have that $D$ is a tensorial body for which $Q_{k_j}^i$ and $(Q_{k_j}^i)^{\circ}$ converge to $D^i$ and $(D^i)^{\circ}$ respectively. 

Now, by the continuity of $\otimes_\pi$ and $\otimes_\epsilon$ (Proposition \ref{prop:porducto projectivo e injectivo es continuo con hauss}), we have that $Q_{k_{j}}^{1}\otimes_{\pi}\cdots\otimes_{\pi}Q_{k_{j}}^{l}$ and $Q_{k_{j}}^{1}\otimes_{\epsilon}\cdots\otimes_{\epsilon}Q_{k_{j}}^{l}$ converge to $D^{1}\otimes_{\pi}\cdots\otimes_{\pi}D^{l}$ and $D^{1}\otimes_{\epsilon}\cdots\otimes_{\epsilon}D^{l},$ respectively. Since we also have that
$
Q_{k_{j}}^{1}\otimes_{\pi}\cdots\otimes_{\pi}Q_{k_{j}}^{l}\subseteq Q_{k_{j}}\subseteq Q_{k_{j}}^{1}\otimes_{\epsilon}\cdots\otimes_{\epsilon}Q_{k_{j}}^{l},
$
then by the Blaschke selection theorem (\cite[Theorem 1.8.6]{Schneider1993}), we can suppose that $Q_{k_j}$ converges to some compact convex set $Q.$ In such case, we must have that $D^{1}\otimes_{\pi}\cdots\otimes_{\pi}D^{l}\subseteq Q\subseteq D^{1}\otimes_{\epsilon}\cdots\otimes_{\epsilon}D^{l}$ and so $Q$ is a tensorial body.
 Finally, from the continuity of $l_\otimes$ (Proposition \ref{prop:lsigma es RETRACCION}), it follows that $l_\otimes(Q)=B_2^{d_1,\ldots,d_l}.$  This shows that $Q\in\mathscr{L}_{\otimes}$ 
 and the proof is completed.
\end{proof}

\begin{thm}
\label{thm:compact global slice}
$\mathscr{L}_{\otimes}\left(\otimes_{i=1}^{l}\mathbb{R}^{d_{i}}\right)$
is a compact global $O_{\otimes}$-slice for the proper $GL_{\otimes}$-space
$\mathcal{B}_{\otimes}\left(\otimes_{i=1}^{l}\mathbb{R}^{d_{i}}\right).$
\end{thm}
\begin{proof}
The compactness of $\mathscr{L}_{\otimes}$
was proved in Proposition \ref{prop:-Lsigma es compacto}. Now, we will prove that it is a  global $GL_\otimes$-slice for the space of tensorial bodies. 
Since $\mathscr{E}_{\otimes}$
is the $GL_{\otimes}$-orbit of $B_{2}^{d_1\ldots,d_l},$ see \cite[Corollary 4.3]{tensorialbodies}, and
$O_{\otimes}$
is the stabilizer of $B_{2}^{d_1\ldots,d_l},$
then, by \cite[Proposition 1.1.5]{Palais2}, there is a $GL_\otimes$-equivariant homeomorphism between $\mathscr{E}_{\otimes}$
and $GL_{\otimes}/O_{\otimes}.$
The latter together with Proposition \ref{prop:lsigma es RETRACCION} and the fact that 
$
\mathscr{L}_{\otimes}(\otimes_{i=1}^{l}\mathbb{R}^{d_{i}})=l_{\otimes}^{-1}(B_{2}^{d_1\ldots,d_l})
$
yield a $GL_{\otimes}$-equivariant map $f:\mathcal{B}_{\otimes}\rightarrow GL_{\otimes}/O_{\otimes}$
such that $\mathscr{L}_{\otimes}(\otimes_{i=1}^{l}\mathbb{R}^{d_{i}})=f^{-1}(O_{\otimes}(\otimes_{i=1}^{l}\mathbb{R}^{d_{i}})).$ 
\end{proof}

\subsection{Topological representatives for $\mathcal{B}_{\otimes}(\otimes_{i=1}^{l}\mathbb{R}^{d_i})$ and $\mathcal{BM}_{\otimes}(\otimes_{i=1}^{l}\mathbb{R}^{d_i})$}
\label{sec:main results}

Here, we exhibit topological representatives for both the space of tensorial bodies and the compactum $\mathcal{BM}_\otimes(\otimes_{i=1}^{l}\mathbb{R}^{d_i})$. Namely, in Corollary \ref{cor:topol rep for BM},  we prove that $\mathcal{BM}_\otimes(\otimes_{i=1}^{l}\mathbb{R}^{d_i})$ is homeomorphic to two orbit spaces $\mathcal{B}_{\otimes}/GL_{\otimes}$ and $\mathscr{L}_{\otimes}/O_{\otimes}.$ In Corollaries \ref{cor:homeomorfismo importante} and \ref{cor: tensorial bodies and Rp}, we prove that $\mathcal{B}_\otimes$ is homeomorphic to the product $\mathcal{\mathscr{L}}_{\otimes}\times\mathscr{E}_{\otimes}$ and $\mathscr{L}_{\otimes}\times\mathbb{R}^p,$ respectively.
It is worth to notice that these results extend, to the context of tensorial bodies, Corollary 1 of \cite{banacamazurcompactum} and Corollaries 3.8  and 3.9 of \cite{Antonyan2013}  about the Banach-Mazur compactum and the space of convex bodies in $\mathbb{R}^d$, respectively. 

\begin{prop}
\label{prop:reprensentative for BM tensor}
The following statements hold:
\begin{enumerate}
\item The orbit
space $\mathcal{B}_{\otimes}(\otimes_{i=1}^{l}\mathbb{R}^{d_{i}})/GL_{\otimes}(\otimes_{i=1}^{l}\mathbb{R}^{d_{i}})$
is compact.
\item The spaces $\mathscr{L}_{\otimes}(\otimes_{i=1}^{l}\mathbb{R}^{d_{i}})/O_{\otimes}(\otimes_{i=1}^{l}\mathbb{R}^{d_{i}})$
and $\mathcal{B}_{\otimes}(\otimes_{i=1}^{l}\mathbb{R}^{d_{i}})/GL_{\otimes}(\otimes_{i=1}^{l}\mathbb{R}^{d_{i}})$
are ho\-meomorphic.
\end{enumerate}
\end{prop}

\begin{proof}
(1). From Proposition \ref{prop:-Lsigma es compacto}, $\mathcal{\mathscr{L}}_{\otimes}$
is compact and $GL_{\otimes}\left(\mathscr{L}_{\otimes}\right)=\mathcal{B}_{\otimes}.$
Thus, by the continuity of the orbit map 
$
\pi:\mathcal{B}_{\otimes}\rightarrow\mathcal{B}_{\otimes}/GL_{\otimes},
$
we have that 
$
\pi\left(\mathscr{L}_{\otimes}\right)=\mathcal{B}_{\otimes}/GL_{\otimes}
$
 is compact.

(2). Denote by $\pi_{\mid}$ the restriction of the orbit map to $\mathscr{L}_{\otimes}.$
From (1), we know that $\pi_{\mid}$ is a continuous surjective map
from $\mathscr{L}_{\otimes}\left(\otimes_{i=1}^{l}\mathbb{R}^{d_{i}}\right)$
onto $\mathcal{B}_{\otimes}/GL_{\otimes}.$ Also, notice that from (3) in Proposition
\ref{prop:-Lsigma es compacto}, for every $P,Q\in\mathscr{L}_{\otimes}$
we have that $\pi_{\mid}\left(P\right)=\pi_{\mid}\left(Q\right)$ if and
only if $P$, $Q$ have the same $O_{\otimes}$-orbit. From this, $\pi_{\mid}$
induces a continuous bijective map 
$
\rho:\mathcal{\mathscr{L}}_{\otimes}/O_{\otimes}\rightarrow\mathcal{B}_{\otimes}/GL_{\otimes}.
$
 Since $\mathscr{L}_{\otimes}$ is compact (Proposition \ref{prop:-Lsigma es compacto}), the same holds for $\mathscr{L}_{\otimes}/O_{\otimes},$ see \cite[Theorem 3.1]{Bredon}. Finally, due to 
$
\mathcal{B}_{\otimes}/GL_{\otimes}
$
 is Hausdorff (see \cite[Proposition 1.1.4]{Palais2}),
we have that $\rho$ is a homeomorphism between $\mathcal{\mathscr{L}}_{\otimes}/O_{\otimes}$ and $\mathcal{B}_{\otimes}/GL_{\otimes}.$ This completes the proof.
\end{proof}

Recall that $\mathcal{BM}_\otimes(\otimes_{i=1}^{l}\mathbb{R}^{d_i})$ endowed with the metric $\log\delta_{\otimes}^{BM}$ is a compact metric space (\cite[Theorem 3.13]{tensorialbodies}). It consists of the equivalence classes of tensorial bodies determined by the relation $P\sim Q$ if and only if $T(P)=Q$ for some $T\in GL_\otimes$, see (\ref{eq:tensorial bm distance}).

\begin{cor}
\label{cor:topol rep for BM}
$\mathcal{BM}_\otimes(\otimes_{i=1}^{l}\mathbb{R}^{d_i})$ is homeomorphic to $\mathcal{B}_{\otimes}(\otimes_{i=1}^{l}\mathbb{R}^{d_{i}})/GL_{\otimes}(\otimes_{i=1}^{l}\mathbb{R}^{d_{i}})$ and $\mathscr{L}_{\otimes}(\otimes_{i=1}^{l}\mathbb{R}^{d_{i}})/O_{\otimes}(\otimes_{i=1}^{l}\mathbb{R}^{d_{i}}).$
\end{cor}

\begin{proof}
Let $\Psi:\mathcal{B}_{\otimes}\left(\otimes_{i=1}^{l}\mathbb{R}^{d_{i}}\right)\rightarrow\mathcal{BM}_{\otimes}\left(\otimes_{i=1}^{l}\mathbb{R}^{d_{i}}\right)$ be the map sending each tensorial body $Q$  to its class $[Q]\in\mathcal{BM}_{\otimes}.$ It is not difficult to prove that, as a consequence of Lemma \ref{lem:convergencia hausdorff implica convegencia de normas} and \cite[Proposition 3.7]{tensorialbodies}, $\Psi$
is continuous. Clearly it is surjective, and $\Psi(P)=\Psi(Q)$ if and only if $Q$ belongs to the $GL_\otimes$-orbit of $P.$ Therefore, by the compactness of $\mathcal{B}_{\otimes}/GL_{\otimes}$ (see (1) in Proposition \ref{prop:reprensentative for BM tensor}) and the fact that $\mathcal{BM}_\otimes$ is a Hausdorff space, it follows that  $\Psi$ induces a homeomorphism between $\mathcal{B}_{\otimes}/GL_{\otimes}$
and $\mathcal{BM}_{\otimes}.$ This proves the first part of the corollary. The second part follows directly from  (2) in Proposition \ref{prop:reprensentative for BM tensor}.
\end{proof}

\begin{cor}
\label{cor:homeomorfismo importante} The following statements hold:
\begin{enumerate}
\item There exists an $O_{\otimes}$-equivariant
retraction 
$
r:\mathcal{B}_{\otimes}\left(\otimes_{i=1}^{l}\mathbb{R}^{d_{i}}\right)\rightarrow\mathcal{\mathscr{L}}_{\otimes}\left(\otimes_{i=1}^{l}\mathbb{R}^{d_{i}}\right)
$
 such that $r\left(P\right)$ belongs to the $GL_{\otimes}$-orbit
of $P.$
\item $\mathcal{B}_{\otimes}\left(\otimes_{i=1}^{l}\mathbb{R}^{d_{i}}\right)$
is homeomorphic to $\mathcal{\mathscr{L}}_{\otimes}\left(\otimes_{i=1}^{l}\mathbb{R}^{d_{i}}\right)\times\mathscr{E}_{\otimes}\left(\otimes_{i=1}^{l}\mathbb{R}^{d_{i}}\right).$
\end{enumerate}
\end{cor}

\begin{proof}
In order to prove the first part of the theorem, we will use Lemma \ref{lem: split thm of Gl tensor}.  There, it is proved that $GL_\otimes$ is homeomorphic to the product $\mathcal{A}\times O_\otimes$ where $\mathcal{A}\subset GL_\otimes$ consists of the tensor products $S_1\otimes\cdots\otimes S_l$ of  strictly positive maps $S_i\in GL(d_i),$ and the homeomorphism is given by the composition map.
\medskip

(1) To define the so called retraction, let us first consider $f:GL_{\otimes}\rightarrow\mathscr{E}_{\otimes}$
defined as
$
f(T):=T(B_{2}^{d_1,\ldots,d_l}).
$
By \cite[Proposition 1.1.5]{Palais2} and \cite[Corollary 4.3]{tensorialbodies}, $f$ induces
a $GL_{\otimes}$-equivariant homeomorphism 
$
\tilde{f}:GL_{\otimes}/O_{\otimes}\rightarrow\mathscr{E}_{\otimes}.
$
 Indeed, $f$ is the composition of the maps: 
\[
GL_{\otimes}\overset{\pi}{\rightarrow}Gl_{\otimes}/O_{\otimes}\overset{\tilde{f}}{\rightarrow}\mathscr{E}_{\otimes},
\]
where $\pi$ is the natural quotient map. 
From the compactness of $O_{\otimes},$ it follows  that $\pi$ is closed (\cite[Theorem 3.1]{Bredon}). 
Also, since $f$ is the composition of two closed maps, it must be closed too. From this and Lemma \ref{lem: split thm of Gl tensor}, we know that the restriction $f_{\mid\mathcal{A}}$
is a homeomorphism between $\mathcal{A}$ and $\mathscr{E}_{\otimes}$. Furthermore, if we let $O_{\otimes}$ acts on $\mathcal{A}$
by sending the pair $(U,S)\in O_{\otimes}\times\mathcal{A}$ to $USU^{-1},$
 and on $\mathscr{E}_{\otimes}$ by the action induced from $\mathcal{B}_{\otimes},$
then $f_{\mid\mathcal{A}}$ is an $O_{\otimes}$-equivariant homeomorphism. 

Denote by $\xi:\mathscr{E}_{\otimes}\rightarrow\mathcal{A}$ the inverse
map of $f_{\mid\mathcal{A}},$ then 
\begin{equation}
\left[\xi\left(\mathcal{E}\right)\right]^{-1}\mathcal{E}=B_{2}^{d_1,\ldots,d_l}\text{ for all }\mathcal{E}\in\mathscr{E}_{\otimes}.\label{eq:caracteristic de XI}
\end{equation}
We claim that the map $r:\mathcal{B}_{\otimes}\rightarrow\mathcal{\mathscr{L}}_{\otimes},$ defined as,
$
r\left(Q\right):=\left[\xi\left(l_{\otimes}\left(Q\right)\right)\right]^{-1}Q
$
is the desired $O_\otimes$-equivariant retraction.
By its definition $r$ is continuous and $r(Q)$ belongs to the $GL_\otimes$-orbit of $Q\in\mathcal{B}_\otimes.$ Also, from (\ref{eq:caracteristic de XI}) and the equivariance
of $l_\otimes$ (Proposition \ref{prop:lsigma es RETRACCION}), we have
\begin{equation*}
l_{\otimes}\left(r(Q)\right) =l_{\otimes}\left(\left[\xi\left(l_{\otimes}\left(Q\right)\right)\right]^{-1}Q\right)=\left[\xi\left(l_{\otimes}\left(Q\right)\right)\right]^{-1}l_{\otimes}\left(Q\right)=B_{2}^{d_1,\ldots,d_l}.
\end{equation*}
This shows that $r(Q)\in\mathscr{L}_{\otimes}$ for all $Q\in\mathcal{B}_\otimes.$ To prove that it is a retraction onto $\mathscr{L}_{\otimes}$, observe that for every $Q\in\mathscr{L}_{\otimes},$ 
\begin{equation*}
r\left(Q\right)=\left[\xi\left(l_{\otimes}\left(Q\right)\right)\right]^{-1}Q=\left[\xi\left(B_{2}^{d_1,\ldots,d_l}\right)\right]^{-1}Q=I_{\otimes_{i=1}^l\mathbb{R}^{d_{i}}}(Q)=Q.
\end{equation*}
To prove that $r$ is $O_{\otimes}$-equivariant, let $U\in O_{\otimes}$
and $Q\in\mathcal{B}_{\otimes}$ then 
\[
r\left(UQ\right)=\left[\xi\left(l_{\otimes}\left(UQ\right)\right)\right]^{-1}UQ=\left[\xi\left(Ul_{\otimes}\left(Q\right)\right)\right]^{-1}UQ.
\]
By the equivariance of $\xi,$ we have $\xi\left(Ul_{\otimes}\left(Q\right)\right)=U\xi\left(l_{\otimes}\left(Q\right)\right)U^{-1}.$
Thus
$
\left[\xi\left(Ul_{\otimes}\left(Q\right)\right)\right]^{-1}=U\left[\xi\left(l_{\otimes}\left(Q\right)\right)\right]^{-1}U^{-1}.
$
 Consequently, 
\[
r\left(UQ\right)=\left(U\left[\xi\left(l_{\otimes}\left(Q\right)\right)\right]^{-1}U^{-1}\right)UQ=U\left(\left[\xi\left(l_{\otimes}\left(Q\right)\right)\right]^{-1}Q\right)=U(r\left(Q\right))
\]
as required.

To prove (2), define $\varphi:\mathcal{B}_{\otimes}\rightarrow\mathcal{\mathscr{L}}_{\otimes}\times\mathscr{E}_{\otimes}$ as $\varphi(Q):=\left(r(Q),l_{\otimes}(Q)\right),$
then $\varphi$ is an $O_{\otimes}$-equivariant homeomorphism with
inverse map given by $\varphi^{-1}\left(Q,\mathcal{E}\right)=\xi\left(\mathcal{E}\right)Q.$ 
\end{proof}
The next corollary follows directly from the above and Corollary \ref{cor:glsigma over osigma and ellipsoids}:
\begin{cor}
\label{cor: tensorial bodies and Rp}
$\mathcal{B}_{\otimes}(\otimes_{i=1}^{l}\mathbb{R}^{d_{i}})$
is homeomorphic to $\mathcal{\mathscr{L}}_{\otimes}(\otimes_{i=1}^{l}\mathbb{R}^{d_{i}})\times\mathbb{R}^p,$ with $p=\frac{d_1(d_1+1)}{2}+\cdots+\frac{d_l(d_l+1)}{2}.$
\end{cor}

\appendix

\section{The Lie group structure of $GL_\otimes(\otimes_{i=1}^l\mathbb{R}^{d_i})$}
\label{sec: properties of GL tensor}

%
%


Proposition 3.11 of \cite{tensorialbodies} shows that $GL_\otimes(\otimes_{i=1}^l\mathbb{R}^{d_i})$ is a closed subgroup of  $GL(\otimes_{i=1}^l\mathbb{R}^{d_i})$, with respect to the topology induced by the operator norm on $\mathcal{L}(\otimes_{H,i=1}^l\mathbb{R}^{d_i}).$ Consequently, $GL_\otimes(\otimes_{i=1}^l\mathbb{R}^{d_i})$ is a Lie group and  $O_\otimes(\otimes_{i=1}^l\mathbb{R}^{d_i})$ is a compact subgroup of it.

Given a permutation $\sigma$ on $\{1,\ldots,l\}$ for which $x^{\sigma(1)}\otimes\cdots\otimes x^{\sigma(l)}\in\otimes_{i=1}^l\mathbb{R}^{d_i},$ whenever $x^i\in\mathbb{R}^{d_i},$ $i=1,\ldots,l,$ we define the map $U_\sigma$ in decomposable vectors as: 
\begin{equation}
\label{eq: orthogonal perm}
U_\sigma(x^{1}\otimes\cdots\otimes x^{l})=x^{\sigma(1)}\otimes\cdots\otimes x^{\sigma(l)},
\end{equation}
and extend it to $\otimes_{i=1}^l\mathbb{R}^{d_i}$ by multilinearity.  The map $U_{\sigma}$ is orthogonal. Indeed, notice that the canonical basis  $e_{k_i}^{d_i},$ $k_i=1,\ldots,d_i,$ of $\mathbb{R}^{d_i}$ is such that
 $U_{\sigma}(e_{k_1}^{d_1}\otimes\cdots\otimes e_{k_l}^{d_l})=e_{k_{\sigma(1)}}^{d_{\sigma(1)}}\otimes\cdots\otimes e_{k_{\sigma(l)}}^{d_{\sigma(l)}}.$ 
We will denote by $\mathcal{P}$ to the subset of $GL_\otimes(\otimes_{i=1}^l\mathbb{R}^{d_i})$ of orthogonal maps $U_\sigma$ as above. 

\begin{lem}
 The following hold:
\label{lem: finite group H}
\begin{enumerate}
\item $\mathcal{P}$ is a finite subgroup of $O_\otimes(\otimes_{i=1}^l\mathbb{R}^{d_i}).$
\item Every $T\in GL_\otimes(\otimes_{i=1}^l\mathbb{R}^{d_i})$ can be written as $T=(T_{1}\otimes\cdots\otimes T_{l})U_\sigma,$ for some $T_i\in GL(d_i),$ $i=1,\ldots,l,$ and $U_\sigma\in\mathcal{P}.$ 
\end{enumerate}
\end{lem}
\begin{proof}
(1) Clearly $\mathcal{P}$ is  a finite subset of $O_\otimes$ and the identity map on $\otimes_{i=1}^l\mathbb{R}^{d_i}$ belongs to it. 
To prove that it is a subgroup, let $U_\sigma, U_\beta\in \mathcal{P},$ then  it can be directly checked that $U_\sigma^{-1}=U_{\sigma^{-1}}$ and $U_\sigma U_\beta=U_{\beta\sigma}.$ 

(2) Let $T\in GL_\otimes$ then, by \cite[Corollary 2.1.4]{limcampocualquiera}, there exist a permutation  $\sigma$ on $\left\{ 1,...,l\right\}$ and $T_{i}\in GL(d_i),$ $i=1,\ldots,l,$ such that $T(x^{1}\otimes\cdots\otimes x^{l})=T_{1}(x^{\sigma(1)})\otimes\cdots\otimes T_{l}(x^{\sigma(l)}).$ Hence, $T=(T_{1}\otimes\cdots\otimes T_{l})U_\sigma$ as desired.
\end{proof}

\begin{prop}
\label{prop: maximal comp group}
$O_{\otimes}(\otimes_{i=1}^l\mathbb{R}^{d_i})$ is a maximal compact subgroup of $GL_{\otimes}(\otimes_{i=1}^l\mathbb{R}^{d_i}).$
\end{prop}

\begin{proof}
Let $K\subset GL_{\otimes}$ be a compact subgroup such that $O_{\otimes}\subseteq K.$ We will prove that $K=O_{\otimes}.$ 

Suppose that $T\in K$ then, by Lemma \ref{lem: finite group H}, there exist $T_i\in GL(d_i),$ $i=1,\ldots,l,$ and $U_\sigma\in\mathcal{P}$ such that $T=T_{1}\otimes\cdots\otimes T_{l}U_\sigma.$ 
By  the polar decomposition (\cite[Theorem 60]{Kaplansky1974a}), each $T_i$ can be written as $T_i=S_iU_i,$ for some positive linear map $S_i\in GL(d_i)$ and $U_i\in O(d_i).$ Therefore, $T=S_{1}\otimes\cdots\otimes S_{l}U_{1}\otimes\cdots\otimes U_{l}U_\sigma$ and so $S_{1}\otimes\cdots\otimes S_{l}$ is a positive self-adjoint linear map in $K.$ 
Now, from the compactness of $K,$ for each eigenvalue $\lambda_i$ of $S_i$ with unitary eigenvector $z^i$ and each integer $n,$ the sequence $(S_{1}\otimes\cdots\otimes S_{l})^n(z^1\otimes\cdots\otimes z^l)=(\lambda_1\cdots\lambda_l)^nz^1\otimes\cdots\otimes z^l$ is bounded in $\otimes_{i=1}^l\mathbb{R}^{d_i}$. Since this is only possible if $\lambda_1\cdots\lambda_l=1,$ and this holds for every eigenvalue $\lambda_1\cdots\lambda_l$ of $S_{1}\otimes\cdots\otimes S_{l},$ we have that $S_{1}\otimes\cdots\otimes S_{l}$ is the identity on $\otimes_{i=1}^l\mathbb{R}^{d_i}.$  Thus, $T\in O_{\otimes}$ and $K\subseteq O_{\otimes}$ as required.
\end{proof}

Let us denote by $\otimes(GL(d_1),\ldots,GL(d_l))$ the set of tensor products  $T_{1}\otimes\cdots\otimes T_{l}$ of linear maps $T_i\in GL(d_i).$  Similarly, $\otimes(O(d_1),\ldots,O(d_l))$  denotes  the set of tensor products $U_{1}\otimes\cdots\otimes U_{l}$ of orthogonal maps  $U_i\in O(d_i).$ 
Below, we show that $\otimes(GL(d_1),\ldots,GL(d_l))$ and $\otimes(O(d_1),\ldots,O(d_l))$ are Lie groups of dimensions $d_1^2+\cdots+d_l^2-(l-1)$ and $\frac{d_1(d_1-1)}{2}+\cdots+\frac{d_1(d_1-1)}{2}-(l-1),$ respectively.

\begin{lem}
\label{lem:charact sigma GL}
Let $N\subset GL(d_1)\times\cdots\times GL(d_l)$ be defined as $N:=\{(\lambda_1I_{d_1},\ldots,\lambda_lI_{d_l}):\lambda_1\cdots\lambda_l=1\}.$
Then $N$ is a normal subgroup and:
\begin{enumerate}
\item $\otimes(GL(d_1),\ldots,GL(d_l))$ is isomorphic as a Lie group to $GL(d_1)\times\cdots\times GL(d_l)/N.$

\item $\otimes(O(d_1),\ldots,O(d_l))$ is isomorphic as a Lie group to $O(d_1)\times\cdots\times O(d_l)/N.$
\end{enumerate}
\end{lem}

\begin{proof}
We begin by proving that both $\otimes(GL(d_1),\ldots,GL(d_l))$ and $\otimes(O(d_1),\ldots,O(d_l))$ are closed subgroups of $GL_\otimes(\otimes_{i=1}^l\mathbb{R}^{d_i})$ and, in consequence, they are Lie groups.  From the properties of the tensor product of linear maps, it follows easily that both of them are subgroups. The closedness follows from two facts. First, they are subsets of the set of decomposable vectors in $\mathcal{L}(\mathbb{R}^{d_1},\mathbb{R}^{d_1})\otimes\cdots\otimes \mathcal{L}(\mathbb{R}^{d_l},\mathbb{R}^{d_l}),$ which  is closed with respect to any norm topology on the tensor space (see \cite[Proposition 4.2]{DeSilva2008}). Second,  $GL_\otimes$ is closed in $GL(\otimes_{i=1}^l\mathbb{R}^{d_i}),$ see \cite[Proposition 3.11]{tensorialbodies}. 

We now construct the desired isomorphisms. Let $\Phi:GL(d_1)\times\cdots\times GL(d_l)\rightarrow\otimes(GL(d_1),\ldots,GL(d_l))$ be the map sending each tuple $(T_1,\ldots,T_l)$ to its tensor product $T_1\otimes\cdots\otimes T_l.$ Also, let $\Phi_|$ the restriction of $\Phi$ to $O(d_1)\times\cdots\times O(d_l).$  Clearly, $\Phi$ and $\Phi_|$ are surjective maps. Indeed, it is not difficult to prove that they are smooth homomorphisms with Kernel $N.$ Thus, by \cite[Theorem 11.1.8]{HilgertNeeb}, $\Phi$ and $\Phi_|$ induce isomorphisms of Lie groups between $GL(d_1)\times\cdots\times GL(d_l)/N$ and $\otimes(GL(d_1),\ldots,GL(d_l)),$ and $O(d_1)\times\cdots\times O(d_l)/N$ and $\otimes(O(d_1),\ldots,O(d_l)),$ respectively.
\end{proof}

Below, we describe the structure of the groups $GL_\otimes(\otimes_{i=1}^l\mathbb{R}^{d_i})$ and $O_{\otimes}(\otimes_{i=1}^l\mathbb{R}^{d_i}).$ The case of tensor products of two spaces, i.e. $l=2,$  was already established in  \cite[Proposition A.1]{Mettler}.
\begin{prop}
\label{prop: charact of GL tensor}
$GL_\otimes(\otimes_{i=1}^l\mathbb{R}^{d_i})$ and $O_{\otimes}(\otimes_{i=1}^l\mathbb{R}^{d_i})$ are isomorphic (as Lie groups) to the semidirect product $\otimes(GL(d_1),\ldots,GL(d_l))\rtimes \mathcal{P}$ 
and $\otimes(O(d_1),\ldots,O(d_l))\rtimes \mathcal{P},$ respectively.
\end{prop}

\begin{proof}
First observe that when only different integers are considered, \textit{i.e.} $d_i\neq d_j,$ for $i\neq j,$ $i,j=1,\ldots,l,$ then, by \cite[Corollary 2.1.4]{limcampocualquiera}, $GL_\otimes(\otimes_{i=1}^l\mathbb{R}^{d_i})=\otimes(GL(d_1),\ldots,GL(d_l)),$ $O_\otimes(\otimes_{i=1}^l\mathbb{R}^{d_i})=\otimes(O(d_1),\ldots,O(d_l))$ and $\mathcal{P}=\{I_{\otimes_{i=1}^{l}\mathbb{R}^{d_i}}\}.$ So in this case, the result is straightforward.

In order to prove the general case, notice that from (2) in Lemma \ref{lem: finite group H}, $GL_\otimes=\otimes(GL(d_1),\ldots,GL(d_l))\mathcal{P}$ and $\otimes(GL(d_1),\ldots,GL(d_l))\cap \mathcal{P}$ is trivial. Therefore, we only need to check that $\otimes(GL(d_1),\ldots,GL(d_l))$ is a closed normal subgroup of $GL_\otimes$, and that $GL_\otimes$  has a finite number of connected components. The first part of the result then follows from \cite[Proposition 11.1.18]{HilgertNeeb}.

By Lemma \ref{lem:charact sigma GL}, $\otimes(GL(d_1),\ldots,GL(d_l))$ is a closed subgroup.
 To prove that it is a normal subgroup, it is enough to show that $U_\sigma(T_{1}\otimes\cdots\otimes T_{l})U_\sigma^{-1}\in\otimes(GL(d_1),\ldots,GL(d_l))$ for any $T_i\in GL(d_i)$ and $U_\sigma\in \mathcal{P}.$ 
Let $x^i\in\mathbb{R}^{d_i},$ $i=1,\ldots,l,$ then
\begin{align}
\nonumber U_\sigma(T_{1}\otimes\cdots\otimes T_{l})U_\sigma^{-1}(x^{1}\otimes\cdots\otimes x^{l}) & =U_\sigma T_{1}(x^{\sigma^{-1}(1)})\otimes\cdots\otimes T_{l}(x^{\sigma^{-1}(l)})\\=T_{\sigma(1)}(x^1)\otimes\cdots\otimes T_{\sigma(l)}(x^l)
&=T_{\sigma(1)}\otimes\cdots\otimes T_{\sigma(l)}(x^1\otimes\cdots\otimes x^l).\label{eq:normal subgroup}
\end{align}
Thus, $U_\sigma(T_{1}\otimes\cdots\otimes T_{l})U_\sigma^{-1}=T_{\sigma(1)}\otimes\cdots\otimes T_{\sigma(l)}$ is a linear map on $\otimes(GL(d_1),\ldots,GL(d_l)),$ as required.

To prove that $GL_\otimes$ has a finite number of connected components, note that by Lemma \ref{lem: finite group H}, the composition induces a continuous surjective map from $\otimes(GL(d_1),\ldots,GL(d_l))\times \mathcal{P}$ onto $GL_\otimes.$ 
Hence, since both $\otimes(GL(d_1),\ldots,GL(d_l))$ and $\mathcal{P}$ have a finite number of connected components, the same holds for $GL_\otimes$.

It remains to show the assertion for $O_\otimes.$ By Lemma \ref{lem:charact sigma GL}, $\otimes(O(d_1),\ldots,O(d_l))$ is closed. Also, since (\ref{eq:normal subgroup}) is valid for $U_i\in O(d_i),$ $i=1,\ldots,l,$ then it is a closed normal subgroup of $O_\otimes.$ In addition, by Lemma \ref{lem: finite group H}, the composition map also gives us a surjective continuous map from $\otimes(O(d_1),\ldots,O(d_l))\times \mathcal{P}$ onto $O_\otimes.$ From this $O_\otimes$ has a finite number of connected components, and the result follows from \cite[Proposition 11.1.18]{HilgertNeeb}.
\end{proof}


\begin{cor}
\label{cor: structure GL/O tensor}
$GL_\otimes(\otimes_{i=1}^l\mathbb{R}^{d_i})/O_{\otimes}(\otimes_{i=1}^l\mathbb{R}^{d_i})$ is homeomorphic to  $\mathbb{R}^p$ with $p=\frac{d_1(d_1+1)}{2}+\cdots+\frac{d_l(d_l+1)}{2}.$
\end{cor}

\begin{proof}
By Proposition \ref{prop: charact of GL tensor}, $GL_{\otimes}$ has a finite number of connected components. Consequently, its
quotient by  $O_{\otimes}$, which is a maximal compact subgroup (Proposition \ref{prop: maximal comp group}), must be homeomorphic to $\mathbb{R}^p,$  \cite[Theorem 32.5]{Stroppel}. Therefore, from Lemma \ref{lem:charact sigma GL} and Proposition \ref{prop: charact of GL tensor}, we have $p=\frac{d_1(d_1+1)}{2}+\cdots+\frac{d_l(d_l+1)}{2}.$
\end{proof}

The next lemma exhibits the polar decomposition of any linear isomorphism in $GL_{\otimes}.$ It shows that every $T\in GL_\otimes$ can be written uniquely as $T=SU$ where $U$ is an orthogonal map in $O_{\otimes}$ and $S$ is the tensor product of strictly positive linear maps. 
Recall that a linear map $S:\mathbb{E}\rightarrow\mathbb{E},$ on a Euclidean space $\mathbb{E},$ is strictly positive if it is self-adjoint and  $\left\langle x,Tx\right\rangle_{\mathbb{E}} >0$ for every non-zero $x\in\mathbb{E}$.

\begin{lem}
\label{lem: split thm of Gl tensor}
Let $\mathcal{A}$ consists of tensor products $S_{1}\otimes\cdots\otimes S_{l}$ of strictly positive linear maps $S_{i}\in GL(d_i),$ $i=1,\ldots,l$. Then:
\begin{enumerate} 
\item $\mathcal{A}\subset GL_{\otimes}(\otimes_{i=1}^{l}\mathbb{R}^{d_{i}})$ is  closed.
\item The map $\Psi:\mathcal{A}\times O_{\otimes}(\otimes_{i=1}^{l}\mathbb{R}^{d_{i}})\rightarrow GL_{\otimes}(\otimes_{i=1}^{l}\mathbb{R}^{d_{i}})$ sending the pair $(S,U)$ to $SU$
is a homeomorphism.
\end{enumerate}
\end{lem}

\begin{proof}
(1).
To show that $\mathcal{A}$ is closed, let $S_{1,n}\otimes\cdots\otimes S_{l,n}$  be a sequence in $\mathcal{A}$ converging to $S\in GL_{\otimes}.$ Since  $\otimes(GL(d_1),\ldots,GL(d_l))$ is closed (Proposition \ref{lem:charact sigma GL}), and positive self-adjoint linear maps are stable under taking limits, then $S$ is a positive linear isomorphism (i.e. strictly positive). Moreover, it is of the form $S=T_1\otimes\cdots\otimes T_l,$ for some $T_i\in GL(d_i),$ $i=1,\ldots,l.$  Now, if $T_{i}=S_{i}U_{i}$ is the polar decomposition of $T_{i}$
(i.e. $S_{i}$ is strictly positive and $U_{i}$
is orthogonal), then $S=(S_{1}\otimes\cdots\otimes S_{l})(U_{1}\otimes\cdots\otimes U_{l}).$ So, by the uniqueness of the polar decomposition (\cite[Theorem 60]{Kaplansky1974a}), $U_{1}\otimes\cdots\otimes U_{l}$ must be the identity on $\otimes_{i=1}^{l}\mathbb{R}^{d_{i}}.$ This shows that $S\in\mathcal{A}$ as desired.


(2). To show that $\Psi$ is bijective, let $T\in GL_\otimes$ then, by Lemma \ref{lem: finite group H}, $T=T_{1}\otimes\cdots\otimes T_{l}U_\sigma$ for some $T_{i}\in GL(d_i)$ and $U_\sigma\in\mathcal{P}.$ Also, as a consequence of the above argument, $T$ can be written as $T=(S_{1}\otimes\cdots\otimes S_{l})(U_{1}\otimes\cdots\otimes U_{l})U_{\sigma}$ and so $T=SU,$ for $S=S_{1}\otimes\cdots\otimes S_{l}\in\mathcal{A}$ and $U=(U_{1}\otimes\cdots\otimes U_{l})U_{\sigma}\in O_\otimes.$ 
Indeed, since the polar decomposition of linear isomorphims is unique (\cite[Theorem 60]{Kaplansky1974a}) then so are $S$, $U$. Hence $\Psi$ is bijective. Its continiuty follows by definition.
It reminds to show that $\Psi^{-1}$ is continuous. Let $T_{n}=\Psi(S_{n},U_{n})$ be a sequence converging to $T=\Psi(S,U)$. From the compactness of $O_{\otimes}$  and the fact that $\mathcal{A}$ is closed, it follows that $U_{n}$ converges to $U$ and $S_n$ to $S$. This shows that $\Psi^{-1}$ is continuous as required.
\end{proof}


\end{document}